\documentclass[11pt,a4paper,leqno]{amsart}

\usepackage[latin1]{inputenc}
\usepackage[T1]{fontenc}
\usepackage{amsfonts}
\usepackage{amsmath}
\usepackage{amssymb}
\usepackage{eurosym}
\usepackage{mathrsfs}
\usepackage{palatino}
\usepackage{color}
\usepackage{xcolor}
\usepackage{esint}
\usepackage[
    colorlinks=true,
    linkcolor=blue,
    citecolor=blue,
    urlcolor=blue
]{hyperref}

\usepackage{amsmath,amssymb,amsthm}
\usepackage{graphicx}
\usepackage{tikz}          % 绘图
\usepackage{pgfplots}      % 专业绘图
\pgfplotsset{compat=1.18}  % 设置pgfplots版本

\numberwithin{equation}{section}

\newcommand{\BMO}[0]{\operatorname{BMO}}

\newcommand{\loc}[0]{\operatorname{loc}}

\renewcommand{\Re}[0]{\operatorname{Re}}

\swapnumbers
\theoremstyle{plain}
\newtheorem{thm}[equation]{Theorem}
\newtheorem{lem}[equation]{Lemma}
\newtheorem{prop}[equation]{Proposition}
\newtheorem{cor}[equation]{Corollary}

\theoremstyle{definition}

\newtheorem{ass}[equation]{Assumption}

\theoremstyle{remark}
\newtheorem{rem}[equation]{Remark}

\title{Sparse bounds and joint oscillation for bilinear rough commutators}
\author{Yuhao Wu}

\address{Center for Applied Mathematics, Tianjin University, Weijin Road 92, 300072 Tianjin, China}
\email{\text{yuhao\_wu@tju.edu.cn}}

\makeatletter
\@namedef{subjclassname@2020}{%
  \textup{2020} Mathematics Subject Classification}
\makeatother


\begin{document}

\allowdisplaybreaks

\begin{abstract}
We study joint oscillation conditions for commutators of bilinear
rough singular integrals with bounded mean-zero angular kernels.
For second-order commutators with two possibly distinct complex-valued
symbols, we prove lower bounds under a geometric non-degeneracy
assumption. These bounds distinguish the commutation positions:
mixed commutators control the product of the separate quadratic
oscillations, while repeated-position commutators of the partial
transposes additionally control the quadratic average of the product
of the centered symbols.

We also establish sparse domination for iterated commutators,
retaining the symbol oscillations inside the local averages.
For each fixed commutator order, the sparse constant is
$O(s/(s-1))$ as the averaging exponent $s$ decreases to $1$.
The resulting joint oscillation conditions yield strong bounds
in the Banach range and weak bounds in the quasi-Banach range.
At the natural exponent triples, the necessary and sufficient
conditions differ by an arbitrarily small increase in the
oscillation exponents.
\end{abstract}

\maketitle

\section{Introduction and main results}

The boundedness of a commutator can encode the
oscillation of its symbol.
The theorem of Coifman, Rochberg, and Weiss
\cite{MR412721} establishes this connection for
the Riesz transforms and $\BMO$, while Bloom's
theorem \cite{MR805955} identifies the corresponding
weighted oscillation in a two-weight setting.
For iterated commutators involving distinct symbols,
however, boundedness need not force each symbol
to belong to $\BMO$.
The relevant conditions must account for the
interaction of their local oscillations.

Hyt\"onen, Li, and Oikari \cite{MR4143395}
developed this viewpoint for iterated commutators
of linear singular integrals.
Their estimates involve both products of separate
oscillation averages and averages of products
of centered symbols, with a small gap between
the exponents in the necessary and sufficient
conditions. The present paper studies this question for
bilinear rough singular integrals.
In this setting, there is an additional issue:
the two commutations may act in different input
positions or in the same position, and these
configurations have different algebraic structures.

Grafakos and Torres \cite{MR1880324} developed
the multilinear Calder\'on--Zygmund theory.
Sharp maximal estimates were obtained in
\cite{MR1979948}.
Multiple-weight inequalities were later established
in \cite{MR2483720}.
For iterated and generalized commutators,
P\'erez, Pradolini, Torres, and Trujillo-Gonz\'alez
\cite{PerezPradoliniTorresTrujillo2014}
and Xue and Yan \cite{XueYan2016}
established weighted and endpoint estimates.
The two-weight theory was further developed by
Kunwar and Ou \cite{KunwarOu2018} and
Li \cite{MR4453692}.
Zeng \cite{zeng2025offdiagonalbloomweightedestimates}
obtained off-diagonal two-weight estimates for
bilinear iterated commutators with a single symbol,
including lower bounds for real-valued symbols
and non-degenerate Calder\'on--Zygmund operators.
Here we study joint oscillation conditions for
two distinct complex-valued symbols and rough
bilinear kernels, with particular attention to
their dependence on the commutation positions.

Let $d\ge1$. We consider
\begin{equation}\label{equ1.1}
T_\Omega(f_1,f_2)(x)
=
\operatorname{p.v.}
\iint_{\mathbb R^d\times\mathbb R^d}
\frac{\Omega((y_1,y_2)/|(y_1,y_2)|)}
     {|(y_1,y_2)|^{2d}}
f_1(x-y_1)f_2(x-y_2)\,dy_1\,dy_2,
\end{equation}
where
\[
\Omega\in L^\infty(S^{2d-1}),
\qquad
\int_{S^{2d-1}}\Omega\,d\sigma=0.
\]
This class goes back to Coifman and Meyer
\cite{MR380244}.
Grafakos, He, and Honz\'ik \cite{MR3758426}
established boundedness for bounded angular
kernels without smoothness assumptions.
Subsequent developments include the multiplier
criteria of Grafakos, He, and Slav\'ikov\'a
\cite{MR4055166}, the improved
bounds of He and Park \cite{HePark2023},
and results near critical angular integrability
in \cite{DosidisSlavikova2024,
DosidisParkSlavikova2026}.
An alternative approach based on local Fourier
series of the inputs was developed by Bhojak
and Shrivastava \cite{BhojakShrivastava2025}.
In dimension one, Dan, Qin, and Xue
\cite{DanQinXue2026} obtained bounds under
a fractional geometric integrability condition.

Commutators with rough kernels also have
a substantial linear theory.
Hu \cite{MR1949046} proved $L^p$ estimates
under weak angular integrability assumptions.
Endpoint estimates were obtained by Lan, Tao,
and Hu \cite{MR4168195} and by Hu and Tao
\cite{MR4543811}.
These results illustrate the role of cancellation
and angular integrability when kernel smoothness
is unavailable.

Sparse domination provides a useful framework
for upper estimates in the rough setting.
The form methods developed in
\cite{MR3668591,MR3873113}
underlie Barron's sparse bounds for bilinear
rough singular integrals \cite{MR3690231}.
Grafakos, Wang, and Xue
\cite{GrafakosWangXue2022}
and Dan and Xue \cite{DanXue2025}
obtained sparse and weighted estimates under
weaker angular integrability assumptions.
Wang, Xue, and Duan \cite{MR4358246}
treated maximal truncations and also established
sparse estimates for bilinear rough commutators.
For commutators, sparse bounds that keep
the symbol oscillations explicit are particularly
useful; see \cite{MR3695871,
KunwarOu2018,LiSparse2018,MR4453692}.
We retain these oscillations inside the local
averages in order to estimate them by joint
quantities, rather than by separate $\BMO$ norms.

The lower estimates require a different analysis.
For bilinear commutators, necessary conditions
have been obtained through kernel tests and
weak factorization; see Chaffee \cite{Chaffee2016},
Li and Wick \cite{LiWick2017},
and Guo, Lian, and Wu \cite{GuoLianWu2020}.
Weighted necessity was studied by Wang
\cite{Wang2023Necessity}.
Hyt\"onen's non-degeneracy approach
\cite{MR4338459} and its bilinear development
by Oikari \cite{MR4647944}
provide further methods for recovering
symbol oscillation from commutator bounds.
For two distinct complex-valued symbols,
the product of the symbol differences has
no fixed sign, and the joint quantities
must be recovered directly.

The distinction between commutation positions
is already visible in the kernels.
Writing $x$ for the output and $y,z$ for
the inputs, mixed commutation produces
\[
(b_1(x)-b_1(y))(b_2(x)-b_2(z)),
\]
whereas two commutations in the first position
produce
\[
(b_1(x)-b_1(y))(b_2(x)-b_2(y)).
\]
The first expression involves two independently
varying input variables; the second couples
both symbols across the same pair of points.
The commutation positions thus lead to different
joint oscillation conditions.

Our lower bounds use a geometric non-degeneracy
condition adapted to tests with two nearby
inputs and a separated output.
The two kernel displacements then have a common
midpoint direction and small opposite perturbations.
We require positivity, after a fixed complex
rotation, uniformly in these perturbations
on a positive-measure set of directions.
This permits measurable angular dependence
without requiring an open cap of positivity.
A density-point argument converts the condition
into uniform integral estimates for the deficit
of the rotated kernel from a positive reference
value.
These estimates make it possible to absorb
the errors in tests for joint oscillation.
Partial transposition gives the corresponding
geometry for repeated commutations.

Under the geometric condition, we obtain a lower bound for the
mixed commutator from $L^2\times L^2$ to $L^1$, and lower bounds
for repeated-position commutators of the partial transposes from
$L^2\times L^\infty$ to $L^2$ and from
$L^\infty\times L^2$ to $L^2$.
The mixed bound controls $S_{2,2}$; the repeated-position bounds
also control $T_2$.

For the upper bounds, we prove sparse domination for iterated
commutators of every fixed order, with constant
$C_{d,k}s/(s-1)$ for $s>1$.
This yields strong estimates in the Banach range and weak
estimates in the quasi-Banach range, under joint oscillation
conditions with an arbitrarily small increase in the exponents.

We first state the two-sided comparison at the
natural exponent triples.
We then present the quantitative sparse estimate
and the general lower and upper bounds from which
this comparison follows.

Given a bilinear operator \(T\), define
\[
[b,T]_1(f_1,f_2)
:=
bT(f_1,f_2)-T(bf_1,f_2)
\]
and
\[
[b,T]_2(f_1,f_2)
:=
bT(f_1,f_2)-T(f_1,bf_2).
\]
For \(i\in\{1,2\}\), set
\[
C_b^{0,i}(T):=T,
\qquad
C_b^{m,i}(T):=[b,C_b^{m-1,i}(T)]_i,
\qquad m\geq1,
\]
and, for \(k_1,k_2\in\mathbb N_0\), write
\begin{equation}\label{equ4.15}
C_b^{k_{\{1,2\}}}(T)
:=
C_b^{k_1,1}\bigl(C_b^{k_2,2}(T)\bigr),
\qquad k=k_1+k_2.
\end{equation}
Commutators taken in different input positions commute, so the order
between the first- and second-position iterations is immaterial.

For two symbols, we use the following joint
oscillation quantities, following
\cite{MR4143395}.
For $1\le r,t,u<\infty$, define
\begin{equation}\label{eq:joint-oscillation-conditions}
\begin{aligned}
S_{r,t}(b_1,b_2)
&:=
\sup_Q
\langle |b_1-\langle b_1\rangle_Q|\rangle_{r,Q}
\langle |b_2-\langle b_2\rangle_Q|\rangle_{t,Q},\\
T_u(b_1,b_2)
&:=
\sup_Q
\langle
 |b_1-\langle b_1\rangle_Q|
 |b_2-\langle b_2\rangle_Q|
\rangle_{u,Q}.
\end{aligned}
\end{equation}
The quantity $S_{r,t}$ is the supremum of the
products of the separate oscillation averages,
whereas $T_u$ measures the product of the centered
symbols in a single average.
Both are nondecreasing in their exponents, and
H\"older's inequality gives $T_u\le S_{2u,2u}$.
By the John--Nirenberg inequality, they are finite
for all finite exponents if both symbols belong
to $\BMO$.
Their finiteness does not, in general, imply
individual $\BMO$ membership.

For a bilinear operator $T$ and two symbols, write
\[
C_{\vec b}^{i,j}(T)
:=[b_1,[b_2,T]_j]_i,
\qquad
(i,j)\in\{(1,2),(1,1),(2,2)\}.
\]
The following two-sided estimates summarize the
main consequences of our necessary and sufficient
conditions at the natural exponent triples.

\begin{cor}[Two-sided joint oscillation estimates]
\label{cor:two-sided-main}
Let $\Omega\in L^\infty(S^{2d-1})$ have mean zero
and satisfy Assumption~\ref{ass:rough-geometric-nondegeneracy}.
Let $b_1,b_2\in L^3_{\loc}(\mathbb R^d)$ be
complex-valued.

For every $\varepsilon>0$,
\begin{equation}\label{eq:mixed-joint-comparison}
\begin{aligned}
S_{2,2}(b_1,b_2)+T_1(b_1,b_2)
&\lesssim_\Omega
\bigl\|C_{\vec b}^{1,2}(T_\Omega)\bigr\|_
 {L^2\times L^2\to L^1}\\
&\lesssim_{d,\varepsilon}
\|\Omega\|_\infty
\Bigl(
S_{2+\varepsilon,2+\varepsilon}(b_1,b_2)
+T_{1+\varepsilon}(b_1,b_2)
\Bigr).
\end{aligned}
\end{equation}

If, in addition, $b_1b_2\in L^2_{\loc}(\mathbb R^d)$,
then
\begin{equation}\label{eq:diagonal-joint-comparison}
\begin{aligned}
S_{2,2}(b_1,b_2)+T_2(b_1,b_2)
&\lesssim_\Omega
\bigl\|C_{\vec b}^{1,1}(T_\Omega^{*2})\bigr\|_
 {L^2\times L^\infty\to L^2}
=
\bigl\|C_{\vec b}^{2,2}(T_\Omega^{*1})\bigr\|_
 {L^\infty\times L^2\to L^2}\\
&\lesssim_{d,\varepsilon}
\|\Omega\|_\infty
\Bigl(
S_{2+\varepsilon,2+\varepsilon}(b_1,b_2)
+T_{2+\varepsilon}(b_1,b_2)
\Bigr).
\end{aligned}
\end{equation}
\end{cor}

The two-sided estimates above follow from the
more general results stated next.
We begin with a quantitative sparse bound for
iterated commutators, which is also used to
derive the sufficient joint oscillation estimates.
\begin{thm}[Quantitative sparse domination for iterated commutators]
\label{equ1.20}
Let \(k_1,k_2\in\mathbb N_0\), \(k=k_1+k_2\), and let
\(T_\Omega\) be given by \eqref{equ1.1}. If \(k\geq1\), assume that
\[
b\in L^{3k}_{\loc}(\mathbb R^d).
\]
Then, for every \(s>1\) and all
\(f_1,f_2,f_3\in L^\infty_c(\mathbb R^d)\),
\begin{equation}\label{eq:higher-commutator-sparse}
\begin{aligned}
&\bigl|
\langle
C_b^{k_{\{1,2\}}}(T_\Omega)(f_1,f_2),f_3
\rangle
\bigr|\\
&\quad\leq
C_{d,s,k}\|\Omega\|_\infty
\sum_{m_1=0}^{k_1}\sum_{m_2=0}^{k_2}
\sup_{\mathcal S}\sum_{Q\in\mathcal S}|Q|
\prod_{j=1}^3
\left\langle
|b-\langle b\rangle_Q|^{m_j}|f_j|
\right\rangle_{s,Q},
\end{aligned}
\end{equation}
where
\[
m_3=k-m_1-m_2,
\]
and the supremum is taken over all sparse collections with a fixed
dimensional sparsity parameter. The constant satisfies
\begin{equation}\label{eq:quantitative-commutator-constant}
C_{d,s,k}
\lesssim_{d,k}
\frac{s}{s-1},
\qquad s>1.
\end{equation}
\end{thm}

We use the homogeneous extension of $\Omega$ to
$\mathbb R^{2d}\setminus\{0\}$.
The following theorem gives necessary joint
oscillation conditions for the three commutation
positions.
\begin{thm}[Lower bounds and partial transposition]
\label{thm:joint-necessary-main}
Let $\Omega\in L^\infty(S^{2d-1})$ have mean zero
and satisfy the geometric non-degeneracy condition
of Assumption~\ref{ass:rough-geometric-nondegeneracy}.
Let $b_1,b_2\in L^2_{\mathrm{loc}}(\mathbb R^d)$
be complex-valued.

\smallskip
\noindent
\textup{(i)}
We have
\begin{equation}\label{eq:S22-rough-lower}
S_{2,2}(b_1,b_2)+T_1(b_1,b_2)
\lesssim_\Omega
\|C_{\vec b}^{1,2}(T_\Omega)\|_
 {L^2\times L^2\to L^1}.
\end{equation}

\smallskip
\noindent
\textup{(ii)}
If additionally
$b_1b_2\in L^2_{\mathrm{loc}}(\mathbb R^d)$, then
\begin{equation}\label{eq:rough-lower-11}
\begin{aligned}
S_{2,2}(b_1,b_2)+T_2(b_1,b_2)
&\lesssim_\Omega
\|C_{\vec b}^{1,1}(T_\Omega^{*2})\|_
 {L^2\times L^\infty\to L^2}\\
&=
\|C_{\vec b}^{2,2}(T_\Omega^{*1})\|_
 {L^\infty\times L^2\to L^2}.
\end{aligned}
\end{equation}
Here the partial transposes are taken with respect
to the bilinear pairing $\langle F,h\rangle=\int Fh$.
The norms are understood as the norms of the
associated trilinear forms, or equivalently of
their corresponding bounded realizations.
The implied constants depend only on the dimension,
the kernel, and the non-degeneracy data.
\end{thm}

The repeated-position terms are two dual
representations of the same trilinear form.
The geometric assumption and its linear
specialization are discussed in
Section~\ref{sec:geometric-nondegeneracy}.
All three lower estimates use the original kernel
with its output separated from its two inputs.

The sufficient estimates hold for every bounded
mean-zero angular kernel and follow from the
two-symbol sparse bounds.

\begin{thm}[Upper bound for mixed commutators]
\label{thmm4.5}
Let \(1<p_1,p_2\leq\infty\) and
$
\frac1p=\frac1{p_1}+\frac1{p_2}>0.
$
Assume that
$
b_1,b_2\in L^3_{\loc}(\mathbb R^d).
$
For every \(\varepsilon>0\),
\[
\begin{aligned}
\|C_{\vec b}^{1,2}(T_\Omega)\|_
 {L^{p_1}\times L^{p_2}\to\mathcal X^p}
&\lesssim_{d,p_1,p_2,\varepsilon}
\|\Omega\|_\infty
\Bigl(
T_{\max\{p+\varepsilon,1+\varepsilon\}}(b_1,b_2)\\
&\qquad+
S_{\max\{p_1'+\varepsilon,p+\varepsilon\},
        \max\{p_2'+\varepsilon,p+\varepsilon\}}(b_1,b_2)
\Bigr),
\end{aligned}
\]
where
\[
\mathcal X^p=
\begin{cases}
L^p, & 1\leq p<\infty,\\
L^{p,\infty}, & 1/2<p<1.
\end{cases}
\]
\end{thm}

For two commutations in the same input position,
the sufficient condition takes the following form.

\begin{thm}[Upper bounds for diagonal commutators]
\label{thm:diagonal-sufficient-main}
Under the assumptions of Theorem~\ref{thmm4.5},
let $i\in\{1,2\}$.
For every \(\varepsilon>0\), set
\[
m_i=\max\{p+\varepsilon,p_i'+\varepsilon\}.
\]
Then
\begin{equation}\label{eq:diagonal-sufficient}
\begin{aligned}
\|C_{\vec b}^{i,i}(T_\Omega)\|_
 {L^{p_1}\times L^{p_2}\to\mathcal X^p}
\lesssim_{d,p_1,p_2,\varepsilon}
\|\Omega\|_\infty
\bigl(
T_{m_i}(b_1,b_2)+S_{m_i,m_i}(b_1,b_2)
\bigr),
\quad i=1,2.
\end{aligned}
\end{equation}
\end{thm}

For the mixed estimate in
Corollary~\ref{cor:two-sided-main}, apply
Theorem~\ref{thmm4.5} with $p_1=p_2=2$ and combine
it with Theorem~\ref{thm:joint-necessary-main}\textup{(i)}.
For the repeated-position estimates, the commutators of the
partial transposes agree with those of principal-value operators
having mean-zero angular kernels $\Omega^{\ast,j}$ satisfying
\[
\|\Omega^{\ast,j}\|_\infty\lesssim_d\|\Omega\|_\infty,
\qquad j=1,2.
\]
Hence Theorem~\ref{thm:diagonal-sufficient-main}
applies to $T_\Omega^{\ast,2}$ with
$(p_1,p_2)=(2,\infty)$ and to $T_\Omega^{\ast,1}$
with $(p_1,p_2)=(\infty,2)$.
In both cases the relevant oscillation exponent is
$2+\varepsilon$.
Together with
Theorem~\ref{thm:joint-necessary-main}\textup{(ii)},
this proves \eqref{eq:diagonal-joint-comparison}.

Only the lower bounds require geometric
non-degeneracy; the upper bounds hold for every
bounded mean-zero angular kernel.

Since $T_1\leq S_{2,2}$, the necessary condition
for the mixed commutator reduces to the finiteness
of $S_{2,2}$.
For repeated commutations, the lower estimates
also detect the stronger quantity $T_2$.
Removing the arbitrarily small increase in the
oscillation exponents remains an open problem.

The paper is organized as follows.
Section~\ref{sec2} introduces the notation,
the truncation framework, and the geometric
non-degeneracy condition.
Section~\ref{sec3.2} proves the quantitative
sparse estimates for iterated commutators.
Section~\ref{sec4} establishes the joint oscillation
lower bounds and the sufficient conditions,
including the weak estimates in the quasi-Banach range.

\subsection*{Acknowledgements}

The author is deeply grateful to Professor Kangwei Li for suggesting
the problem and for his valuable guidance and discussions. The author
also thanks his mentors and colleagues for their helpful comments and
support.

\subsection*{AI use Statement.}
All mathematical content and proofs in this paper are solely due to the author. 
Large language models assisted with language editing and
proof checking during manuscript preparation.

\section{Preliminaries}\label{sec2}

We fix the notation and sparse framework,
justify passage from smooth truncations to
the limiting forms, and introduce the geometric
non-degeneracy assumption used in Section~\ref{sec4.1}.

\subsection{Notation and sparse families}
\label{sec:notation-sparse}

Throughout the paper, all cubes in \(\mathbb R^d\) have sides parallel
to the coordinate axes. For a cube \(Q\), let \(\ell(Q)\) denote its
side length, and let \(\lambda Q\) be the cube with the same center and
side length \(\lambda\ell(Q)\). We write
\[
\langle f\rangle_Q:=\frac1{|Q|}\int_Q f
\]
and, for \(1\leq r<\infty\),
\[
\langle |f|\rangle_{r,Q}
:=
\left(\frac1{|Q|}\int_Q|f|^r\right)^{1/r},
\qquad
\langle |f|\rangle_{\infty,Q}
:=
\operatorname*{ess\,sup}_{Q}|f|.
\]
For \(1\leq r\leq\infty\), \(r'\) denotes the conjugate exponent. Our
pairing convention is
\[
\langle F,h\rangle:=\int_{\mathbb R^d}F(x)h(x)\,dx,
\]
without complex conjugation. 

For \(0<p<\infty\), the weak \(L^p\) quasi-norm is
\[
\|F\|_{L^{p,\infty}}
:=
\sup_{\lambda>0}
\lambda
\bigl|\{x\in\mathbb R^d:|F(x)|>\lambda\}\bigr|^{1/p}.
\]
We write \(A\lesssim_{\alpha}B\) if
\(A\leq C_\alpha B\), where the constant depends only on the displayed
parameters, and write \(A\simeq_\alpha B\) when both inequalities hold.

In the lower estimates, commutator norms denote
the least constants in the corresponding trilinear
form bounds on bounded compactly supported tests
whose symbol-weighted factors are bounded;
the norm is infinite if no finite bound exists.
The lower estimates therefore apply to every
bounded realization agreeing with these forms.
Norm identities under transposition are understood
in this sense.

For comparison with the upper estimates, we use
the realizations constructed in Section~\ref{sec4}.
Bounds involving an $L^\infty$ input are first
proved for bounded compactly supported inputs,
uniformly in their supports.
The extension to general $L^\infty$ inputs,
weak-$\ast$ continuous at the level of scalar
pairings, is justified at the end of the proof
of Theorem~\ref{thmm4.5}.

A dyadic lattice \(\mathcal D\) is a collection of cubes such that,
for every \(k\in\mathbb Z\), the cubes in \(\mathcal D\) of side length
\(2^{-k}\) form a partition of \(\mathbb R^d\), and any two cubes in
\(\mathcal D\) are either disjoint or one contains the other. We use
the standard adjacent-lattice fact that every cube \(Q\) is contained
in a cube \(R\) belonging to one of finitely many dyadic lattices and
satisfying
\[
\ell(R)\lesssim_d\ell(Q);
\]
see, for example, \cite{MR4007575}.

A collection \(\mathcal S\) of cubes is called \(\eta\)-sparse,
\(0<\eta<1\), if for every \(Q\in\mathcal S\) there is a measurable set
\(E_Q\subset Q\) such that
\[
|E_Q|\geq\eta|Q|
\]
and the sets \(\{E_Q:Q\in\mathcal S\}\) are pairwise disjoint. When the
value of \(\eta\) is irrelevant, we simply say that \(\mathcal S\) is
sparse. 
%For \(Q\in\mathcal S\), we write \(\operatorname{ch}_{\mathcal S}(Q)\) for the maximal cubes in \(\mathcal S\) that are properly contained in \(Q\).

For \(1\leq r<\infty\), define
\[
M_r f(x)
:=
\sup_{Q\ni x}\langle |f|\rangle_{r,Q}.
\]
We write \(M=M_1\). We use the standard boundedness of \(M_r\) on
\(L^p(\mathbb R^d)\) whenever \(1\leq r<p\leq\infty\).

\subsection{Smooth truncations and limiting forms}
\label{sec:truncations-limits}

Let $\chi\in C^\infty([0,\infty))$ be nonincreasing,
with $0\leq\chi\leq1$, and satisfy
\[
\chi(t)=1\quad\text{for }0\le t\le1/4,
\qquad
\chi(t)=0\quad\text{for }t\ge1/2.
\]
For integers $\mu<\nu$, define
\[
\begin{aligned}
T_{\Omega,\mu}^{\nu}(f_1,f_2)(x)
:={}&
\iint_{\mathbb R^d\times\mathbb R^d}
\bigl[
\chi(2^{-\nu}|(y_1,y_2)|)
-\chi(2^{-\mu}|(y_1,y_2)|)
\bigr]\\
&\quad\times
\frac{\Omega((y_1,y_2)/|(y_1,y_2)|)}
{|(y_1,y_2)|^{2d}}
f_1(x-y_1)f_2(x-y_2)\,dy_1\,dy_2,
\end{aligned}
\]
and set
\[
\Lambda_\mu^\nu(f_1,f_2,f_3)
=
\langle T_{\Omega,\mu}^{\nu}(f_1,f_2),f_3\rangle.
\]
When $\nu\le\mu$, we define both
$T_{\Omega,\mu}^{\nu}$ and $\Lambda_\mu^\nu$ to be zero.

These truncations are finite sums of annular pieces
obtained using smooth radial cutoffs,
since
\[
\chi(2^{-\nu}r)-\chi(2^{-\mu}r)
=
\sum_{\mu<a\le\nu}
\bigl[\chi(2^{-a}r)-\chi(2^{1-a}r)\bigr].
\]
On the support of the piece indexed by $a$, all pairwise
distances between the three spatial variables are at most
$2^a$. All finite truncations below use this fixed scale
decomposition.
The following lemma defines the limiting commutator
forms and justifies working with finitely many
scales in the sparse arguments.

\begin{lem}[Passage to the untruncated form]
\label{lem:rough-form-limit}
Let $\Omega\in L^\infty(S^{2d-1})$ have mean zero,
and let $1<q_1,q_2,q_3<\infty$ satisfy
$
\frac1{q_1}+\frac1{q_2}+\frac1{q_3}=1.
$
For every \(h_j\in L^{q_j}(\mathbb R^d)\), \(j=1,2,3\), the limit
\begin{equation}\label{eq:untruncated-form-definition}
\Lambda_\Omega(h_1,h_2,h_3)
:=
\lim_{\substack{\mu\to-\infty\\ \nu\to+\infty}}
\Lambda_\mu^\nu(h_1,h_2,h_3)
\end{equation}
exists and satisfies
\begin{equation}\label{eq:rough-form-bound}
\sup_{\mu<\nu}
|\Lambda_\mu^\nu(h_1,h_2,h_3)|
+
|\Lambda_\Omega(h_1,h_2,h_3)|
\lesssim_{d,q_1,q_2,q_3}
\|\Omega\|_\infty
\prod_{j=1}^3\|h_j\|_{L^{q_j}}.
\end{equation}
For smooth compactly supported inputs, \(\Lambda_\Omega\) agrees with
the principal-value form associated with \(T_\Omega\).

If $k=k_1+k_2\ge1$ and $b\in L^{3k}_{\loc}$,
the finite algebraic expansion of
$C_b^{k_{\{1,2\}}}(T_{\Omega,\mu}^{\nu})$
therefore converges termwise on bounded compactly
supported test functions and defines the corresponding
untruncated commutator form.
The same conclusion holds for
$[b_1,[b_2,T_\Omega]_j]_i$, $i,j\in\{1,2\}$,
when $b_1,b_2\in L^3_{\loc}$.
For bounded compactly supported inputs, these
two-symbol commutator forms are represented by
locally integrable functions.
\end{lem}

\begin{proof}
For $h_j\in C_c^\infty$, set
\[
\Phi(y_1,y_2)
=
\int_{\mathbb R^d}
h_1(x-y_1)h_2(x-y_2)h_3(x)\,dx.
\]
Then $\Phi\in C_c^\infty(\mathbb R^{2d})$.
Writing $y=(y_1,y_2)$, cancellation gives
\[
\begin{aligned}
\Lambda_\mu^\nu(h_1,h_2,h_3)
={}
\int_{\mathbb R^{2d}}
[\chi(2^{-\nu}|y|)-\chi(2^{-\mu}|y|)]
\frac{\Omega(y/|y|)}{|y|^{2d}}
[\Phi(y)-\Phi(0)\mathbf1_{\{|y|\le1\}}]\,dy.
\end{aligned}
\]
Since $\Phi(y)-\Phi(0)=O(|y|)$ near zero
and $\Phi$ is compactly supported, the integrand
without the cutoff factor is absolutely integrable.
The cutoff factor lies in $[0,1]$ and tends to $1$
for $y\ne0$, so dominated convergence gives the
limit and its agreement with the radial
principal-value form.

Choose $1<\rho<\min_jq_j$.
The finite-scale sparse estimate in
\cite[Theorem~1 and Section~4.2]{MR3690231},
followed by the sparse embedding, yields
\[
\begin{aligned}
\sup_{\mu<\nu}
|\Lambda_\mu^\nu(h_1,h_2,h_3)|
&\lesssim_{d,\rho}
\|\Omega\|_\infty
\sup_{\mathcal S}
\sum_{Q\in\mathcal S}|Q|
\prod_{j=1}^3\langle|h_j|\rangle_{\rho,Q}\\
&\lesssim_{d,\vec q}
\|\Omega\|_\infty
\prod_{j=1}^3\|h_j\|_{q_j}.
\end{aligned}
\]
The last step follows by bounding the sparse sum
by $\int\prod_jM_\rho h_j$ and applying
H\"older's inequality and $\rho<q_j$.

For general $h_j\in L^{q_j}$, choose
$h_j^{(n)}\in C_c^\infty$ converging in $L^{q_j}$.
Multilinearity and the uniform bound above imply
\[
\sup_{\mu<\nu}
\left|
\Lambda_\mu^\nu(h_1,h_2,h_3)
-
\Lambda_\mu^\nu(h_1^{(n)},h_2^{(n)},h_3^{(n)})
\right|
\longrightarrow0.
\]
Since each smooth triple has a limit, the original
truncated forms are Cauchy as
$\mu\to-\infty$ and $\nu\to+\infty$.
Passing to the limit proves
\eqref{eq:rough-form-bound}.

For a commutator of total order $k$,
every slot in its expansion is $b^mf_j$,
where $0\le m\le k$.
These functions lie in $L^3$ when
$b\in L^{3k}_{\loc}$ and $f_j\in L^\infty_c$,
so the exponent triple $(3,3,3)$ gives
termwise convergence.
For two symbols in $L^3_{\loc}$,
use $(3,3,3)$ when the symbols occupy different
slots and $(3/2,6,6)$, with the appropriate
permutation, when they occupy the same slot.

Finally, for bounded compactly supported inputs,
the unweighted bounds give
\[
\begin{aligned}
T_\Omega(f_1,f_2)&\in L^3,\\
T_\Omega(b_jf_1,f_2),\
T_\Omega(f_1,b_jf_2)&\in L^{3/2},\\
T_\Omega(b_1f_1,b_2f_2)&\in L^{3/2},\\
T_\Omega(b_1b_2f_1,f_2),\
T_\Omega(f_1,b_1b_2f_2)&\in L^{6/5}.
\end{aligned}
\]
Since $b_j\in L^3_{\loc}$ and
$b_1b_2\in L^{3/2}_{\loc}$, multiplying by
the remaining output symbols gives an
$L^1_{\loc}$ function in every term.
These functions represent the limiting forms.
\end{proof}

\begin{rem}[Truncation convention]
\label{rem:truncation-convention}
Commutators are understood in the form sense of
Lemma~\ref{lem:rough-form-limit}.
In the sparse arguments, $T$ and $\Lambda$ denote
finite smooth scale truncations and their forms.
Local forms retain the upper cutoff at the top-cube
scale, and the recursion stops at the lower cutoff.
The estimates are uniform in both cutoffs and pass
to the limiting forms by Lemma~\ref{lem:rough-form-limit},
since the sparse bounds are independent of the cutoffs.
\end{rem}

\subsection{Geometric non-degeneracy}
\label{sec:geometric-nondegeneracy}

For $\Omega\in L^\infty(S^{2d-1})$, extended
homogeneously of degree zero, write
\[
K_\Omega(x,y,z)
=
\frac{\Omega(x-y,x-z)}{|(x-y,x-z)|^{2d}}.
\]

\begin{ass}[Geometric non-degeneracy]
\label{ass:rough-geometric-nondegeneracy}
There exist a measurable set $F\subset S^{d-1}$
with $\sigma(F)>0$, constants $\delta,\kappa>0$,
and $|\lambda|=1$ such that
\begin{equation}\label{eq:rough-geometric-nondegeneracy}
\operatorname{Re}
[\lambda\Omega(\theta+h,\theta-h)]
\ge\kappa
\quad\text{for a.e. }
(\theta,h)\in F\times B(0,\delta).
\end{equation}
Here $B(0,\delta)\subset\mathbb R^d$.
\end{ass}

The condition describes two nearby inputs and
a separated output. Indeed, setting
\[
\theta=\frac{x-(y+z)/2}{|x-(y+z)/2|},
\qquad
h=\frac{z-y}{2|x-(y+z)/2|}
\]
gives
\[
(x-y,x-z)
=
\left|x-\frac{y+z}{2}\right|
(\theta+h,\theta-h).
\]
Thus $\theta$ is the direction from the input
midpoint to the output, while $|h|$ measures
the relative input separation.
Positivity is uniform in the small displacement $h$;
the direction set $F$ need not contain an open cap,
and $\Omega$ need not be continuous.
\begin{rem}[Examples]
\label{rem:nondegeneracy-examples}
Assumption~\ref{ass:rough-geometric-nondegeneracy}
holds if $\Omega$ has a representative continuous
at a point
\[
\omega_0=2^{-1/2}(\theta_0,\theta_0),
\qquad \theta_0\in S^{d-1},
\]
with $\Omega(\omega_0)\ne0$.
Indeed, choose
$\lambda=\overline{\Omega(\omega_0)}/|\Omega(\omega_0)|$.
Continuity gives
$\operatorname{Re}(\lambda\Omega)\ge
|\Omega(\omega_0)|/2$
in a neighborhood of $\omega_0$.
The assumption follows by taking $F$ to be
a sufficiently small spherical cap around
$\theta_0$ and choosing $\delta>0$ small enough.

The assumption also includes discontinuous
mean-zero kernels.
To see this, let $F\subset S^{d-1}$ be compact
with positive surface measure, fix $0<\delta<1$,
and set
\[
\Phi(\theta,h)
=
\frac{(\theta+h,\theta-h)}
     {\sqrt{2(1+|h|^2)}},
\qquad
E=\Phi\bigl(F\times B(0,\delta)\bigr).
\]
The map $\Phi$ is a smooth parametrization of
the part of $S^{2d-1}$ where $u+v\ne0$,
with inverse
\[
\theta=\frac{u+v}{|u+v|},
\qquad
h=\frac{u-v}{|u+v|}.
\]
Consequently,
$0<\sigma(E)<\sigma(S^{2d-1})$;
the strict upper bound also follows from
$u\cdot v>0$ on $E$.
Define
\[
\Omega_E
=
\mathbf1_E
-
\frac{\sigma(E)}
     {\sigma(S^{2d-1}\setminus E)}
\mathbf1_{S^{2d-1}\setminus E}.
\]
Then $\Omega_E$ is bounded, has mean zero,
and its homogeneous extension satisfies
\[
\Omega_E(\theta+h,\theta-h)=1,
\qquad
\theta\in F,\quad |h|<\delta.
\]
Thus Assumption~\ref{ass:rough-geometric-nondegeneracy}
holds with $\lambda=\kappa=1$.
For $d\ge2$, one may choose $F$ with empty
interior; then $E$ also has empty interior.

Unlike in the linear homogeneous setting, nonvanishing
of $\Omega$ alone does not imply
Assumption~\ref{ass:rough-geometric-nondegeneracy}.
For example, a nonzero mean-zero angular kernel may
vanish in a neighborhood of all diagonal directions
$2^{-1/2}(\theta,\theta)$.
\end{rem}

\begin{rem}[Lebesgue points and the linear case]
\label{rem:rough-lebesgue-linear}
The bounded measurable function
\[
m_{\lambda,\delta}(\theta)
=
\operatorname*{ess\,inf}_{|h|<\delta}
\operatorname{Re}
[\lambda\Omega(\theta+h,\theta-h)]
\]
gives an equivalent formulation of
Assumption~\ref{ass:rough-geometric-nondegeneracy}:
for some $\lambda,\delta$, it has a Lebesgue point
$\theta_0$ with positive Lebesgue value.
Indeed, the assumption gives
$m_{\lambda,\delta}\ge\kappa$ almost everywhere on $F$.
Conversely, a positive Lebesgue value $c$ implies
that $\{m_{\lambda,\delta}\ge c/2\}$ has density one
at $\theta_0$ and hence positive measure.
For $d=1$, use counting measure on $S^0$.

In the linear homogeneous setting, the displacement
variable is absent, and the condition becomes
$\operatorname{Re}[\lambda\Omega]\ge\kappa$
on a positive-measure set.
By restricting the modulus and phase, this is
equivalent to $\Omega\not\equiv0$ almost everywhere,
and hence to the existence of a Lebesgue point
$\theta_0$ with nonzero Lebesgue value.
This recovers the homogeneous non-degeneracy
condition in
\cite[Definition~2.1.1(2)]{MR4338459}.
Choose the representative of $\Omega$ whose value
at $\theta_0$ equals its nonzero Lebesgue value.
For $K(x,y)=\Omega(x-y)/|x-y|^d$, taking
$x=y+2r\theta_0$ gives
\[
|x-y|=2r,
\qquad
|K(x,y)|=\frac{|\Omega(\theta_0)|}{(2r)^d}.
\]
The bilinear condition additionally requires
uniformity in the relative input displacement.
\end{rem}

The integral defect estimates following from
Assumption~\ref{ass:rough-geometric-nondegeneracy}
are proved in Lemma~\ref{lem:rough-integral-smallness},
at the beginning of Section~\ref{sec4.1}.

\section{A sparse domination principle}\label{sec3.2}

We prove Theorem~\ref{equ1.20} using the localized framework
of Barron \cite{MR3690231} and the frequency decomposition
of Grafakos, He, and Honz\'ik \cite{MR3758426}.
Proposition~\ref{prop:local-rough-estimate} proves
$A_{d,s}\lesssim_d s/(s-1)$ by grouping the frequency pieces
before local interpolation. The grouped kernels have uniform
size bounds and retain the required operator-norm decay.

The commutator argument consists of applying the localized
estimate simultaneously to the symbol-weighted inputs.
At each cube, the symbol is recentered at its local mean,
while the remainder is exactly a sum of commutators on
the stopping children. We give this recursion in detail
for first-order commutators and then extend it by the
binomial identity.

\subsection{The localized rough form}

We use the truncation convention of
Remark~\ref{rem:truncation-convention}.
The local spaces and stopping collections below follow
Barron's notation.

Let \(\mathcal P\) be a stopping collection with top \(Q\): its cubes are
pairwise disjoint dyadic cubes contained in \(3Q\), and, writing
\(s_L=\log_2\ell(L)\), they satisfy
\begin{equation}\label{eq:stopping-geometry}
\begin{split}
 |s_L-s_R|\geq8&\quad\Longrightarrow\quad 7L\cap7R=\emptyset,\\
 \bigcup_{\substack{L\in\mathcal P\\3L\cap2Q\neq\emptyset}}9L
 &\subset \operatorname{sh}\mathcal P
 :=\bigcup_{L\in\mathcal P}L .
\end{split}
\end{equation}
Set \(\widehat L=2^5L\) and \(M_p h=(M(|h|^p))^{1/p}\).  For functions
supported in \(3Q\), define
\[
\|h\|_{\mathcal Y_p(\mathcal P)}
:=
\max\left\{
\|h\mathbf1_{\mathbb R^d\setminus\operatorname{sh}\mathcal P}\|_\infty,
\sup_{L\in\mathcal P}\inf_{x\in\widehat L}M_ph(x)
\right\},\qquad p<\infty,
\]
and \(\|h\|_{\mathcal Y_\infty(\mathcal P)}=\|h\|_\infty\).
The space \(\dot{\mathcal X}_p(\mathcal P)\) consists of functions
\(\beta=\sum_{L\in\mathcal P}\beta_L\), where
\(\operatorname{supp}\beta_L\subset L\), \(\int_L\beta_L=0\), and its norm
is the \(\mathcal Y_p(\mathcal P)\)-norm.  These norms increase with \(p\).
Moreover, the stopping-time Calder\'on--Zygmund decomposition
\begin{equation}\label{eq:stopping-cz}
 h=g+\beta,\qquad
 \beta=\sum_{L\in\mathcal P}
 \bigl(h-\langle h\rangle_L\bigr)\mathbf1_L
\end{equation}
satisfies
\begin{equation}\label{eq:stopping-cz-norm}
 \|g\|_{\mathcal Y_\infty(\mathcal P)}
 +\|\beta\|_{\dot{\mathcal X}_p(\mathcal P)}
 \lesssim_d \|h\|_{\mathcal Y_p(\mathcal P)}.
\end{equation}

For a dyadic cube \(L\), let
\[
\Lambda_L(h_1,h_2,h_3)
:=
\Lambda_\mu^{\min\{s_L,\nu\}}
 (h_1\mathbf1_L,h_2\mathbf1_{3L},h_3\mathbf1_{3L}),
\]
and define
\begin{equation}\label{eq:localized-form}
\Lambda_{\mathcal P}(h_1,h_2,h_3)
:=\Lambda_Q(h_1,h_2,h_3)
 -\sum_{\substack{L\in\mathcal P\\L\subset Q}}
  \Lambda_L(h_1,h_2,h_3).
\end{equation}

\begin{prop}[Localized estimate]\label{prop:local-rough-estimate}
For every $s>1$ there is a constant $A_{d,s}$, independent of
the truncation parameters and the stopping collection, such that
\begin{equation}\label{eq:local-rough}
|\Lambda_{\mathcal P}(h_1,h_2,h_3)|
\le A_{d,s}\|\Omega\|_\infty |Q|
\prod_{i=1}^3\|h_i\|_{\mathcal Y_s(\mathcal P)}.
\end{equation}
The constants may be chosen so that
\begin{equation}\label{eq:quantitative-As}
1\le A_{d,s}\lesssim_d\frac{s}{s-1},
\qquad s>1.
\end{equation}
\end{prop}

\begin{proof}
By homogeneity, it suffices to assume
$\|\Omega\|_\infty=1$ and $1<s\le2$.
We follow the localized argument of
\cite[Sections~3--4]{MR3690231}, keeping track of
the constants after grouping the frequency pieces.

Use the decomposition
\[
T_\Omega=\sum_{j\in\mathbb Z}T^j,
\qquad
K_j=\sum_{i\in\mathbb Z}K^i*\psi_{i-j},
\]
from \cite[Section~2]{MR3758426}, where $K^i$ is
supported on $|W|\simeq2^i$ and
\[
\|K^i\|_\infty\lesssim_d2^{-2di},
\qquad
\|K^i\|_1\lesssim_d1.
\]
The frequency estimates in
\cite[Propositions~3 and~5]{MR3758426},
also recorded in \cite[Proposition~4.1]{MR3690231},
give
\begin{equation}\label{eq:frequency-operator-decay}
\|T^j\|_{L^2\times L^2\to L^1}
\lesssim_d2^{-c|j|}
\end{equation}
for some dimensional $c>0$.
Thus the series can be regrouped in operator norm.
Set
\[
I_n=\{j:2^n\le |j|<2^{n+1}\},
\qquad
\widetilde T_n=\sum_{j\in I_n}T^j,
\qquad
\widetilde K_n=\sum_{j\in I_n}K_j.
\]
Then
\[
T_\Omega=T^0+\sum_{n\ge0}\widetilde T_n,
\qquad
\|\widetilde T_n\|_{L^2\times L^2\to L^1}
\lesssim_d2^{-c2^n}.
\]

To avoid a factor counting the pieces in $I_n$,
choose the cutoffs in telescoping form,
\[
\psi_t=\Phi_t-\Phi_{t+1},
\qquad
\Phi_t(W)=2^{-2dt}\Phi(2^{-t}W),
\]
where $\widehat\Phi$ is a smooth radial cutoff
equal to one near the origin.
Put
\[
H_r=\sum_iK^i*\Phi_{i-r}.
\]
The annular support of $K^i$ and the Schwartz
bounds for $\Phi$ imply
\begin{equation}\label{eq:frequency-size-gradient}
|H_r(W)|\lesssim_d|W|^{-2d},
\qquad
|\nabla H_r(W)|
\lesssim_d2^{\max\{r,0\}}|W|^{-2d-1}.
\end{equation}
Indeed, for $r\ge0$ the scale-$i$ bounds are
\[
|\nabla^m(K^i*\Phi_{i-r})(W)|
\lesssim_d
2^{mr}2^{-(2d+m)i}
(1+2^{-i}|W|)^{-2d-2},
\qquad m=0,1.
\]
For $r<0$, use $\|K^i\|_1\lesssim_d1$ at
the smoothing scale $2^{i-r}$.
Summing the resulting bounds proves
\eqref{eq:frequency-size-gradient}.

Since $K_j=H_j-H_{j-1}$, we have
\[
\widetilde K_n
=
H_{2^{n+1}-1}-H_{2^n-1}
+H_{-2^n}-H_{-2^{n+1}}.
\]
Consequently,
\[
|\widetilde K_n(W)|\lesssim_d|W|^{-2d},
\]
and, for $|V|\le |W|/2$,
\begin{equation}\label{eq:frequency-modulus}
\begin{aligned}
|\widetilde K_n(W-V)-\widetilde K_n(W)|
&\lesssim_d |W|^{-2d}
\min\left\{
1,2^{2^{n+1}}\frac{|V|}{|W|}
\right\}.
\end{aligned}
\end{equation}
The kernel $K_0=H_0-H_{-1}$ has dimensional
size and smoothness constants.

These estimates are compatible with the fixed
smooth spatial truncations.
Indeed, multiplication of a convolution kernel
by $\varphi$ increases its
$L^2\times L^2\to L^1$ operator norm by at most
$\|\widehat\varphi\|_1$, by Fourier inversion
and modulation invariance.
For the difference of dilates of $\chi$ used here,
this norm is uniformly bounded.
Let $\widetilde\Lambda_{\mathcal P}^{n}$
denote the localized form of $\widetilde T_n$.
The truncated operator bound and the stopping
geometry therefore give
\[
\begin{aligned}
|\widetilde\Lambda_{\mathcal P}^{n}(u_1,u_2,u_3)|
\le A_2(n)|Q|
\|u_1\|_{\mathcal Y_2}
\|u_2\|_{\mathcal Y_2}
\|u_3\|_{\mathcal Y_\infty},
\quad
A_2(n)\lesssim_d2^{-c2^n}.
\end{aligned}
\]

For the cancellation endpoint, apply the argument
of \cite[Proposition~3.3]{MR3690231} with
\eqref{eq:frequency-modulus}.
At separation $\ell$, the kernel difference and
the annular cutoff contribute
\[
\omega_n(\ell)
=
\min\{1,2^{2^{n+1}-\ell}\}+2^{-\ell}.
\]
Since
$
\sum_{\ell\ge1}\omega_n(\ell)\lesssim2^n,
$
that argument yields
\[
\begin{aligned}
|\widetilde\Lambda_{\mathcal P}^{n}(\beta,g,h)|
\le A_1(n)|Q|
\|\beta\|_{\dot{\mathcal X}_1}
\|g\|_{\mathcal Y_p}
\|h\|_{\mathcal Y_p},
\quad
A_1(n)\lesssim_d2^n,
\quad 1<p\le2.
\end{aligned}
\]
The constant is uniform in $p$.
Indeed, for a stopping cube $L$ and scale
$a=s_L+\ell$, the remaining variables are contained
in a cube $R_{L,\ell}\supset\widehat L$ with
$|R_{L,\ell}|\simeq_d2^{da}$.
For $v=g,h$,
\[
\langle|v|\rangle_{1,R_{L,\ell}}
\le \langle|v|\rangle_{p,R_{L,\ell}}
\le \inf_{x\in\widehat L}M_pv(x)
\le \|v\|_{\mathcal Y_p}.
\]
Thus cancellation and \eqref{eq:frequency-modulus}
bound the contribution of $L$ at this scale by
\[
C_d\omega_n(\ell)|L|\,
\|\beta\|_{\dot{\mathcal X}_1}
\|g\|_{\mathcal Y_p}\|h\|_{\mathcal Y_p},
\]
after a fixed dimensional adjustment of
$\omega_n$; the finitely many near scales follow
from the size bound.
Summing over the disjoint stopping cubes and using
$\sum_{\ell\ge1}\omega_n(\ell)\lesssim_d2^n$
gives the claimed bound, uniformly for $1<p\le2$.

For cancellation in the other positions, use the
representations in
\cite[Proposition~2.1 and Appendix~A]{MR3690231}.
Their boundary terms satisfy the same uniform estimate
by the single-scale size bound, and the decay estimate
applies to the full localized form.

We now apply the threshold argument of
\cite[Lemmas~4.3--4.4]{MR3690231}, recording its
uniform dependence on the parameters.
Set
\[
\delta=\frac{s-1}{16},
\qquad p=1+4\delta,
\qquad q=1+8\delta<s.
\]
For cancellation in the first position, normalize
\[
\|u_1\|_{\dot{\mathcal X}_q}
=\|u_2\|_{\mathcal Y_q}
=\|u_3\|_{\mathcal Y_s}=1.
\]
Split each input at height $\lambda\geq 1$, writing
\[
u_i=u_i^{\mathrm{lo}}+u_i^{\mathrm{hi}},
\]
and subtract the average of $u_1^{\mathrm{hi}}$ on each
stopping cube. Thus both $u_1^{\mathrm{lo}}$ and
$u_1^{\mathrm{hi}}$ have the required cancellation.
The elementary truncation estimates give
\[
\|u_1^{\mathrm{hi}}\|_{\dot{\mathcal X}_1}
   \lesssim_d \lambda^{1-q},
\qquad
\|u_2^{\mathrm{hi}}\|_{\mathcal Y_p}
   \lesssim_d \lambda^{1-q/p},
\qquad
\|u_3^{\mathrm{hi}}\|_{\mathcal Y_p}
   \lesssim_d \lambda^{1-s/p},
\]
\[
\|u_1^{\mathrm{lo}}\|_{\dot{\mathcal X}_2}
   +\|u_2^{\mathrm{lo}}\|_{\mathcal Y_2}
   \lesssim_d \lambda^{1-q/2},
\qquad
\|u_3^{\mathrm{lo}}\|_\infty\leq\lambda.
\]
All other norms needed at the cancellation
endpoint are bounded dimensionally.
These estimates use only pointwise truncation
inequalities and Jensen's inequality, so their
constants remain bounded as $\delta\to0$.

The terms containing a high part can be grouped as
\[
\widetilde{\Lambda}_{\mathcal P}^n
   (u_1^{\mathrm{hi}},u_2,u_3)
+\widetilde{\Lambda}_{\mathcal P}^n
   (u_1^{\mathrm{lo}},u_2^{\mathrm{hi}},u_3)
+\widetilde{\Lambda}_{\mathcal P}^n
   (u_1^{\mathrm{lo}},u_2^{\mathrm{lo}},u_3^{\mathrm{hi}}).
\]
The cancellation endpoint bounds their sum in
absolute value by
\[
C_dA_1(n)|Q|
\bigl(\lambda^{1-q}
+\lambda^{1-q/p}
+\lambda^{1-s/p}\bigr)
\lesssim_d A_1(n)|Q|\lambda^{1-q/p}.
\]
The remaining all-low term
\[
\widetilde{\Lambda}_{\mathcal P}^n
  (u_1^{\mathrm{lo}},u_2^{\mathrm{lo}},u_3^{\mathrm{lo}})
\]
is bounded in absolute value by
$C_dA_2(n)|Q|\lambda^{3-q}$.
Hence
\[
|\widetilde\Lambda_{\mathcal P}^{n}(u_1,u_2,u_3)|
\lesssim_d |Q|
\bigl(A_1(n)\lambda^{1-q/p}
+A_2(n)\lambda^{3-q}\bigr).
\]
Choose endpoint upper bounds with $A_1(n)\ge A_2(n)>0$
and take
\[
\lambda=
\left(\frac{A_1(n)}{A_2(n)}\right)^{1/(2-q+q/p)}.
\]
The resulting exponent satisfies
\[
\alpha=\frac{q/p-1}{2-q+q/p}
=\frac{2\delta}{1+2\delta-16\delta^2}
\ge\delta.
\]
Restoring the input norms, we obtain
\[
\begin{aligned}
|\widetilde\Lambda_{\mathcal P}^{n}(u_1,u_2,u_3)|
\lesssim_d{}&
2^{n(1-\delta)}2^{-c\delta2^n}|Q|
\|u_1\|_{\dot{\mathcal X}_q}
\|u_2\|_{\mathcal Y_q}
\|u_3\|_{\mathcal Y_s}.
\end{aligned}
\]
Interchanging the first two inputs gives the
second-position estimate.
For cancellation in the third position, apply the same
decomposition with the third input in $\dot{\mathcal X}_s$
and the first two in $\mathcal Y_q$. Recenter
$u_3^{\mathrm{hi}}$ on each stopping cube, so that both
$u_3^{\mathrm{lo}}$ and $u_3^{\mathrm{hi}}$ retain
cancellation. Moreover,
\[
\|u_3^{\mathrm{hi}}\|_{\dot{\mathcal X}_1}
   \lesssim_d \lambda^{1-s},
\qquad
\|u_3^{\mathrm{lo}}\|_{\dot{\mathcal X}_1}
   \lesssim_d 1,
\qquad
\|u_3^{\mathrm{lo}}\|_\infty
   \lesssim_d \lambda.
\]
Using cancellation in the third position for the
terms containing a high part gives the same bound
$C_dA_1(n)|Q|\lambda^{1-q/p}$.
The all-low term is again bounded by
$C_dA_2(n)|Q|\lambda^{3-q}$, so the same optimization
applies.
The form associated with $T^0$ is treated with
dimensional endpoint constants.

Finally,
\[
\sum_{n\ge0}
2^{n(1-\delta)}2^{-c\delta2^n}
\le
\sum_{n\ge0}2^n2^{-c\delta2^n}
\lesssim_d\delta^{-1}.
\]
Indeed, the sum over $2^n\le\delta^{-1}$ is
$O(\delta^{-1})$.
If $n_0$ is the first remaining index, then
$1<\delta2^{n_0}\le2$, and the tail is at most
\[
2\delta^{-1}
\sum_{r\ge0}2^r2^{-c2^r}
\lesssim_d\delta^{-1}.
\]
Summation and monotonicity of the local norms give
\begin{equation}\label{eq:local-tests}
\begin{aligned}
|\Lambda_{\mathcal P}(\beta_1,u_2,u_3)|
&\lesssim_d(s-1)^{-1}|Q|
\|\beta_1\|_{\dot{\mathcal X}_s}
\|u_2\|_{\mathcal Y_s}
\|u_3\|_{\mathcal Y_s},
\\
|\Lambda_{\mathcal P}(g_1,\beta_2,u_3)|
&\lesssim_d(s-1)^{-1}|Q|
\|g_1\|_{\mathcal Y_\infty}
\|\beta_2\|_{\dot{\mathcal X}_s}
\|u_3\|_{\mathcal Y_s},
\\
|\Lambda_{\mathcal P}(g_1,g_2,\beta_3)|
&\lesssim_d(s-1)^{-1}|Q|
\|g_1\|_{\mathcal Y_\infty}
\|g_2\|_{\mathcal Y_\infty}
\|\beta_3\|_{\dot{\mathcal X}_s}.
\end{aligned}
\end{equation}
These estimates extend from bounded inputs by
amplitude truncation, with recentering on each
stopping cube for cancellation inputs.
The local norms remain uniformly bounded by
\eqref{eq:stopping-cz-norm}, and convergence in
$L^1(3Q)$ permits passage to the limit for each
fixed spatial truncation.

Apply \eqref{eq:stopping-cz} to write
$h_i=g_i+\beta_i$.
The all-good term is bounded by
$C_d|Q|\prod_i\|g_i\|_\infty$.
The remaining terms are
\[
\Lambda_{\mathcal P}(\beta_1,h_2,h_3)
+\Lambda_{\mathcal P}(g_1,\beta_2,h_3)
+\Lambda_{\mathcal P}(g_1,g_2,\beta_3),
\]
and are controlled by \eqref{eq:local-tests}
and \eqref{eq:stopping-cz-norm}.
This proves
$A_{d,s}\lesssim_d(s-1)^{-1}$ for $1<s\le2$.
For $s>2$, use the estimate at $s=2$ and
monotonicity.
Restoring $\|\Omega\|_\infty$ proves the proposition.
\end{proof}

We shall also use the following standard consequence of the Whitney selection
in Barron's stopping construction.  If \(\mathcal H\) is a finite family of 
$L^s$ functions supported in \(3Q\), then there is a stopping collection
\(\mathcal P(Q)\) satisfying \eqref{eq:stopping-geometry} such that
\begin{equation}\label{eq:stopping-control}
\sum_{\substack{L\in\mathcal P(Q)\\L\subset Q}}|L|
\leq \frac12|Q|,
\qquad
\|h\|_{\mathcal Y_s(\mathcal P(Q))}
\lesssim_{d,\#\mathcal H}
\langle |h|\rangle_{s,3Q}
\quad(h\in\mathcal H).
\end{equation}
Let $\mathcal H$ consist of functions in $L^s(3Q)$, extended by
zero outside $3Q$. Define the global exceptional set
\begin{equation}\label{eq:exceptional-set}
E_Q=\left\{x\in\mathbb R^d:
\max_{\substack{h\in\mathcal H\\
                 \langle|h|\rangle_{s,3Q}>0}}
\frac{M_s h(x)}{\langle|h|\rangle_{s,3Q}}>A\right\}.
\end{equation}
If all functions vanish, take $E_Q=\emptyset$.
The weak type $(1,1)$ estimate for $M$ gives
\[
|E_Q|\le C_d(\#\mathcal H)A^{-s}|3Q|.
\]
Choose $A=A(d,\#\mathcal H)\ge1$ sufficiently large that
\[
|E_Q|\le\delta_d|Q|,
\qquad
0<\delta_d\le\min\{1/2,(9/100)^d\}.
\]
This choice is uniform for $s>1$.

Let $\mathcal W_Q$ be the maximal dyadic cubes $L$ satisfying
$9L\subset E_Q$, and set
\[
\mathcal P(Q)=\{L\in\mathcal W_Q:L\subset3Q\}.
\]
The cubes in $\mathcal W_Q$ are pairwise disjoint and cover $E_Q$
up to a null set. Moreover,
\[
9^d|L|\le |E_Q|\le\delta_d|Q|,
\qquad \ell(L)\le\ell(Q)/100.
\]
Since $3Q$ is a union of dyadic cubes of side length $\ell(Q)$,
each such $L$ meeting $3Q$ is contained in $3Q$, up to boundaries.
Thus
\[
\operatorname{sh}\mathcal P(Q)=E_Q\cap3Q
\quad\text{up to a null set}.
\]
The usual Whitney size comparison gives the first condition in
\eqref{eq:stopping-geometry}. If $3L\cap2Q\ne\emptyset$,
the bound $\ell(L)\le\ell(Q)/100$ implies $9L\subset3Q$.
Hence $9L\subset E_Q\cap3Q$, which gives the second condition.

Let $L^{(1)}$ be the dyadic parent of $L\in\mathcal P(Q)$.
Maximality in the global Whitney collection gives a point
\[
x_L\in9L^{(1)}\setminus E_Q\subset\widehat L.
\]
Consequently, for every nonzero $h\in\mathcal H$,
\[
\inf_{x\in\widehat L}M_s h(x)
\le M_s h(x_L)\le A\langle|h|\rangle_{s,3Q}.
\]
Outside the shadow, the same bound holds almost everywhere for
$|h|$, by Lebesgue differentiation. Therefore
\[
\|h\|_{\mathcal Y_s(\mathcal P(Q))}
\lesssim_{d,\#\mathcal H}\langle|h|\rangle_{s,3Q}.
\]
Finally,
\[
\sum_{\substack{L\in\mathcal P(Q)\\L\subset Q}}|L|
\le |E_Q|\le\frac12|Q|.
\]

\subsection{First-order commutators}

\begin{lem}\label{lem:first-order-commutator-sparse}
Let \(b\in L^3_{\loc}(\mathbb R^d)\) and \(s>1\). For bounded,
compactly supported \(f_1,f_2,f_3\),
\begin{equation}\label{eq:first-commutator-sparse}
\begin{aligned}
|\langle[b,T_\Omega]_1(f_1,f_2),f_3\rangle|
\lesssim A_{d,s}\|\Omega\|_\infty\Bigg[
&\sup_{\mathcal S}\sum_{Q\in\mathcal S}|Q|
 \langle|f_1|\rangle_{s,Q}\langle|f_2|\rangle_{s,Q}
 \langle|b-\langle b\rangle_Q||f_3|\rangle_{s,Q}\\
{}+&\sup_{\mathcal S}\sum_{Q\in\mathcal S}|Q|
 \langle|b-\langle b\rangle_Q||f_1|\rangle_{s,Q}
 \langle|f_2|\rangle_{s,Q}\langle|f_3|\rangle_{s,Q}
\Bigg].
\end{aligned}
\end{equation}
The analogous estimate for \([b,T_\Omega]_2\) has the oscillation in
the second input in the second sparse form.
\end{lem}

\begin{proof}
It suffices to prove the estimate for $1<s\le2$.
Indeed, the estimate at $s=2$ implies the estimate for every
$s>2$ by monotonicity of normalized averages, after enlarging
the dimensional implicit constant.
Assume $1<s\le2$ and use the truncation convention.
Choose a top cube $Q$ containing the support of $f_1$,
with $3Q$ containing the other supports, and large
enough to contain all retained scales.
Put $c_Q=\langle b\rangle_{3Q}$.
All stopping inputs below belong to $L^s(3Q)$ because
$b\in L^3_{\loc}$.

Let $\mathcal C_Q^1$ be the first-position commutator
form associated with $\Lambda_Q$.
Since a commutator is
unchanged when a constant is subtracted from \(b\),
\begin{equation}\label{eq:first-local-algebra}
\mathcal C_Q^1
=\Lambda_Q(f_1,f_2,(b-c_Q)f_3)
 -\Lambda_Q((b-c_Q)f_1,f_2,f_3).
\end{equation}

Apply \eqref{eq:stopping-control} to the five functions
\[
\mathcal H_Q=
\{f_1,f_2,f_3,(b-c_Q)f_1,(b-c_Q)f_3\}\mathbf1_{3Q}.
\]

Using the same constant \(c_Q\) in the identity for every child cube
\(L\in\mathcal P(Q)\), \(L\subset Q\), and subtracting those child
identities from \eqref{eq:first-local-algebra}, we obtain the exact formula
\begin{equation}\label{eq:commutator-localization}
\begin{aligned}
\mathcal C_Q^1 -\sum_{\substack{L\in\mathcal P(Q)\\L\subset Q}}
\mathcal C_L^1
={}&\Lambda_{\mathcal P(Q)}(f_1,f_2,(b-c_Q)f_3)\\
&-\Lambda_{\mathcal P(Q)}((b-c_Q)f_1,f_2,f_3).
\end{aligned}
\end{equation}
By \eqref{eq:stopping-control} and
Proposition~\ref{prop:local-rough-estimate},
\begin{equation}\label{eq:first-local-bound}
\begin{aligned}
\left|\mathcal C_Q^1-\sum_{\substack{L\in\mathcal P(Q)\\L\subset Q}}
\mathcal C_L^1\right|
\lesssim A_{d,s}\|\Omega\|_\infty |Q|\Big(&
\langle|f_1|\rangle_{s,3Q}\langle|f_2|\rangle_{s,3Q}
\langle|b-c_Q||f_3|\rangle_{s,3Q}\\
&+\langle|b-c_Q||f_1|\rangle_{s,3Q}
\langle|f_2|\rangle_{s,3Q}\langle|f_3|\rangle_{s,3Q}\Big).
\end{aligned}
\end{equation}

Iterate \eqref{eq:first-local-bound}, taking
$c_L=\langle b\rangle_{3L}$ at each child $L$.
The recursion stops at the lower cutoff and produces
a $1/2$-sparse tree $\mathcal T$.
Choose pairwise disjoint sets $F_Q\subset Q$ with
$|F_Q|\ge |Q|/2$. Since
\[
F_Q\subset3Q,
\qquad
|F_Q|\ge\frac{1}{2\cdot3^d}|3Q|,
\]
the collection $\{3Q:Q\in\mathcal T\}$ is
$(2\cdot3^d)^{-1}$-sparse.
Relabeling the dilated cubes gives
\eqref{eq:first-commutator-sparse}.

The second-position commutator is treated in the same way.
\end{proof}

\subsection{Higher-order commutators}

\begin{proof}[Proof of Theorem~\ref{equ1.20}]
For $k=0$, apply
Proposition~\ref{prop:local-rough-estimate} and
\eqref{eq:stopping-control} with stopping data
\[\{f_1\mathbf1_{3Q},f_2\mathbf1_{3Q},f_3\mathbf1_{3Q}\}\]
at each top cube $Q$.
The finite-scale recursion and dilation argument in the
preceding proof give
\[
|\Lambda_\mu^\nu(f_1,f_2,f_3)|
\lesssim_d A_{d,s}\|\Omega\|_\infty
\sup_{\mathcal S}\sum_{Q\in\mathcal S}|Q|
\prod_{i=1}^3\langle|f_i|\rangle_{s,Q}.
\]
The truncation convention and
\eqref{eq:quantitative-As} give the assertion.
We may therefore assume that $k\ge1$.
It suffices to prove the estimate for $1<s\le2$.
Indeed, the estimate at $s=2$ implies the result for $s>2$
by monotonicity of normalized averages.

Use the truncation convention and choose a top cube
$Q$ as in the proof of
Lemma~\ref{lem:first-order-commutator-sparse}.
Put $c_Q=\langle b\rangle_{3Q}$.
Since $b\in L^{3k}_{\loc}$, every function
$(b-c_Q)^m f_j\mathbf1_{3Q}$ with $0\le m\le k$
belongs to $L^s(3Q)$ for $1<s\le2$.

Let $\mathcal C_Q^{k_1,k_2}$ denote the iterated
commutator form associated with $\Lambda_Q$.
Expanding
$(b(x)-b(y_1))^{k_1}(b(x)-b(y_2))^{k_2}$
around $c_Q$ gives the exact identity
\begin{equation}\label{eq:binomial-expansion}
\begin{aligned}
\mathcal C_Q^{k_1,k_2}=\sum_{m_1=0}^{k_1}\sum_{m_2=0}^{k_2}&
(-1)^{m_1+m_2}\binom{k_1}{m_1}\binom{k_2}{m_2}\\
&\quad\times
\Lambda_Q\bigl((b-c_Q)^{m_1}f_1,(b-c_Q)^{m_2}f_2,
               (b-c_Q)^{k-m_1-m_2}f_3\bigr).
\end{aligned}
\end{equation}
Use \eqref{eq:stopping-control} with
\begin{equation}\label{eq:higher-stopping-data}
\begin{aligned}
\mathcal H_Q^{(k)}={}&
\{(b-c_Q)^m f_1\mathbf1_{3Q}:0\le m\le k_1\}\\
&\cup\{(b-c_Q)^m f_2\mathbf1_{3Q}:0\le m\le k_2\}\\
&\cup\{(b-c_Q)^m f_3\mathbf1_{3Q}:0\le m\le k\}.
\end{aligned}
\end{equation}
Apply \eqref{eq:binomial-expansion} both to the localized commutator on
\(Q\) and to those on its
children, always with the same \(c_Q\), and subtract.  Exactly as in
\eqref{eq:commutator-localization}, the difference is
\begin{equation}\label{eq:higher-localization}
\begin{aligned}
\sum_{m_1=0}^{k_1}\sum_{m_2=0}^{k_2}
(-1)^{m_1+m_2}\binom{k_1}{m_1}\binom{k_2}{m_2}
\Lambda_{\mathcal P(Q)}
\bigl((b-c_Q)^{m_1}f_1,(b-c_Q)^{m_2}f_2,
      (b-c_Q)^{k-m_1-m_2}f_3\bigr).
\end{aligned}
\end{equation}
Equations \eqref{eq:local-rough}, \eqref{eq:stopping-control}, and
\eqref{eq:higher-stopping-data} bound \eqref{eq:higher-localization} by
\begin{equation}\label{eq:higher-local-bound}
C_{d,s,k}\|\Omega\|_\infty |Q|
\sum_{m_1=0}^{k_1}\sum_{m_2=0}^{k_2}
\prod_{j=1}^3
\langle|b-c_Q|^{m_j}|f_j|\rangle_{s,3Q},
\qquad m_3=k-m_1-m_2.
\end{equation}

The stopping threshold is uniform in $s$, and the
number of stopping inputs and the binomial
coefficients depend only on $k$.
Thus the constant in \eqref{eq:higher-local-bound}
is at most $C_{d,k}A_{d,s}$.
The subsequent sparse iteration and dilation introduce
only constants depending on $d$ and $k$.
Consequently, \eqref{eq:quantitative-As} gives
\eqref{eq:quantitative-commutator-constant}.

Iterating \eqref{eq:higher-local-bound}, with
$c_L=\langle b\rangle_{3L}$ at each new top cube,
produces a $1/2$-sparse tree by
\eqref{eq:stopping-control}.
Its threefold dilation is
$(2\cdot3^d)^{-1}$-sparse, as in the first-order proof.
After relabeling the dilated cubes, each nonnegative
sum indexed by $(m_1,m_2)$ is bounded by the corresponding
supremum in \eqref{eq:higher-commutator-sparse}.
Passing to the untruncated form by
Lemma~\ref{lem:rough-form-limit} completes the proof.
\end{proof}

\section{Joint oscillation conditions for bilinear commutators}
\label{sec4}

For mixed commutators, four tests produce a nonnegative weight
from the two symbol differences. The integral defect estimates
then give the lower bound.
For repeated-position commutators of the partial transposes,
averaging the kernel in the separated variable and using four
joint-oscillation tests recovers both $S_{2,2}$ and $T_2$.
The sufficient conditions follow from sparse forms retaining
the symbol oscillations inside the local averages.

\subsection{Necessary conditions}\label{sec4.1}

Throughout this subsection,
$\Omega\in L^\infty(S^{2d-1})$ has mean zero
and satisfies
Assumption~\ref{ass:rough-geometric-nondegeneracy}.
We first derive uniform integral defect estimates,
then apply them to the mixed and repeated
commutation positions.

\begin{lem}[Integral defect estimates]
\label{lem:rough-integral-smallness}
Under Assumption~\ref{ass:rough-geometric-nondegeneracy},
there exist $\theta_0\in S^{d-1}$, a constant $A_0$,
and numbers $\varepsilon_A\ge0$, with
$\varepsilon_A\to0$ as $A\to\infty$, such that the following
holds for every $A\ge A_0$ and every ball $B=B(y_0,r)$.

Set
\[
 x_0=y_0+Ar\theta_0,\qquad
 \widetilde B=B(x_0,r),\qquad
 k_A=\frac{\kappa}{2(\sqrt2 Ar)^{2d}}.
\]
Then
\begin{equation}\label{eq:rough-integral-smallness}
\begin{aligned}
 &\underset{(y,z)\in B\times B}{\operatorname{ess\,sup}}
 \fint_{\widetilde B}
 \bigl(k_A-\operatorname{Re}(\lambda K_\Omega(x,y,z))\bigr)_+
 \,dx
 \le\varepsilon_A k_A,\\
 &\underset{(x,z)\in\widetilde B\times B}
 {\operatorname{ess\,sup}}
 \fint_B
 \bigl(k_A-\operatorname{Re}(\lambda K_\Omega(x,y,z))\bigr)_+
 \,dy
 \le\varepsilon_A k_A,\\
 &\underset{(x,y)\in\widetilde B\times B}
 {\operatorname{ess\,sup}}
 \fint_B
 \bigl(k_A-\operatorname{Re}(\lambda K_\Omega(x,y,z))\bigr)_+
 \,dz
 \le\varepsilon_A k_A.
\end{aligned}
\end{equation}
Here $t_+=\max\{t,0\}$. The choices of $\theta_0$, $A_0$,
and $\varepsilon_A$ depend only on the fixed kernel and its
non-degeneracy data, and are independent of $B$.
\end{lem}

\begin{proof}
For $d\ge2$, the Lebesgue density theorem on the sphere
provides $\theta_0\in F$ such that
\[
 \lim_{t\rightarrow0}
 \frac{\sigma(F\cap B(\theta_0,t))}
 {\sigma(S^{d-1}\cap B(\theta_0,t))}=1.
\]
Such a point is called a density point of $F$. Since small
spherical caps have measure comparable to $t^{d-1}$,
\[
 \sigma\bigl((S^{d-1}\cap B(\theta_0,t))
 \setminus F\bigr)=o(t^{d-1}).
\]
Set $\widehat F=\{\rho\theta:\rho>0,\ \theta\in F\}$.
If $\rho\theta\in B(\theta_0,t)$ and $0<t<1/2$, then
\[
\begin{aligned}
 |\rho-1|
 &=\bigl||\rho\theta|-|\theta_0|\bigr|
 \le|\rho\theta-\theta_0|<t,\\
 |\theta-\theta_0|
 \le|\theta-&\rho\theta|+|\rho\theta-\theta_0|
 =|1-\rho|+|\rho\theta-\theta_0|<2t.
\end{aligned}
\]
Thus polar coordinates yield
\[
\begin{aligned}
 |B(\theta_0,t)\setminus\widehat F|
 \le
 \int_{1-t}^{1+t}\rho^{d-1}\,d\rho\,
 \sigma\bigl((S^{d-1}\cap B(\theta_0,2t))
 \setminus F\bigr)
 =o(t^d).
\end{aligned}
\]
Dividing by $|B(\theta_0,t)|$ gives
\begin{equation}\label{eq:rough-density-point}
 \fint_{B(\theta_0,t)}
 \mathbf1_{\mathbb R^d\setminus\widehat F}
 \longrightarrow0
 \qquad(t\rightarrow0).
\end{equation}
For $d=1$, choose any $\theta_0\in F\subset\{-1,1\}$;
then $B(\theta_0,t)\subset\widehat F$ for $t<1$,
so \eqref{eq:rough-density-point} also holds.

Choose $A_0>4$ sufficiently large that
$
 (A_0-2)^{-1}<\delta,
 (1+2/A_0)^{2d}\le2.
$
Both inequalities then hold with $A_0$ replaced by
any $A\ge A_0$.
Fix $A\ge A_0$, let $B=B(y_0,r)$, and set
\[
 \widetilde B=B(y_0+Ar\theta_0,r).
\]
In particular, $B$ and $\widetilde B$ are separated.

For $(x,y,z)\in\widetilde B\times B\times B$, set
\[
 \theta=
 \frac{x-(y+z)/2}{|x-(y+z)/2|},
 \qquad
 h=\frac{z-y}{2|x-(y+z)/2|}.
\]
Since
\[
 |h|\le\frac1{A-2},
 \qquad
 \sqrt2(A-2)r
 \le |(x-y,x-z)|
 \le \sqrt2(A+2)r,
\]
our choice of $A_0$ ensures that $|h|<\delta$ and
$(1+2/A)^{2d}\le2$.
By homogeneity,
\[
 \Omega(x-y,x-z)=\Omega(\theta+h,\theta-h).
\]
Consequently, \eqref{eq:rough-geometric-nondegeneracy} gives
\[
 \operatorname{Re}\bigl[\lambda K_\Omega(x,y,z)\bigr]
 \ge
 \frac{\kappa}{[\sqrt2(A+2)r]^{2d}}
 \ge k_A
\]
for almost every triple satisfying
$x-(y+z)/2\in\widehat F$.
Here the almost-everywhere assertion follows from the smooth
coordinates
\[
 (s,\theta,h)\longmapsto
 \bigl(s(\theta+h),s(\theta-h)\bigr),
 \qquad s>0,
\]
on the set where the sum of the two displacements is nonzero.
Using the size bound for $K_\Omega$ on the remaining triples,
we obtain
\begin{equation}\label{eq:rough-defect-support}
 (k_A-\operatorname{Re}(\lambda K_\Omega(x,y,z)))_+
 \le C_d\left(1+\frac{\|\Omega\|_\infty}{\kappa}\right)k_A
 \mathbf1_{\mathbb R^d\setminus\widehat F}
       \left(x-\frac{y+z}{2}\right).
\end{equation}

Write
\[
 x=x_0+ru,\qquad
 y=y_0+rv,\qquad
 z=y_0+rw,
 \qquad u,v,w\in B(0,1).
\]
Since $\widehat F$ is a cone,
\[
 \mathbf1_{\mathbb R^d\setminus\widehat F}
 \left(x-\frac{y+z}{2}\right)
 =
 \mathbf1_{\mathbb R^d\setminus\widehat F}
 \left(
 \theta_0+\frac1A\left(u-\frac{v+w}{2}\right)
 \right).
\]
Since $u,v,w\in B(0,1)$,
\[
 \left|u-\frac{v+w}{2}\right|
 \le |u|+\frac{|v|+|w|}{2}<2.
\]
Thus the argument of the indicator belongs to
$B(\theta_0,2/A)$, regardless of which two variables
are held fixed. We now estimate the three averages separately.

First fix $y,z\in B$, or equivalently $v,w\in B(0,1)$.
Under $x=x_0+ru$, the factor $r^d$ in $dx$ cancels
with the same factor in $|\widetilde B|$.
For the subsequent change of variables
\[
 q=\theta_0+\frac1A\left(u-\frac{v+w}{2}\right),
 \qquad du=A^d\,dq,
\]
the image of $B(0,1)$ is
$B(\theta_0-(v+w)/(2A),1/A)$, which is contained in
$B(\theta_0,2/A)$. Hence
\[
\begin{aligned}
 &\fint_{\widetilde B}
 \mathbf1_{\mathbb R^d\setminus\widehat F}
 \left(x-\frac{y+z}{2}\right)\,dx\\
 &\quad=
 \frac1{|B(0,1)|}
 \int_{B(0,1)}
 \mathbf1_{\mathbb R^d\setminus\widehat F}
 \left(\theta_0+\frac1A
 \left(u-\frac{v+w}{2}\right)\right)\,du\\
 &\quad=
 \frac{A^d}{|B(0,1)|}
 \int_{B(\theta_0-(v+w)/(2A),\,1/A)}
 \mathbf1_{\mathbb R^d\setminus\widehat F}(q)\,dq\\
 &\quad\le
 \frac{A^d}{|B(0,1)|}
 \int_{B(\theta_0,2/A)}
 \mathbf1_{\mathbb R^d\setminus\widehat F}(q)\,dq\\
 &\quad=
 2^d\fint_{B(\theta_0,2/A)}
 \mathbf1_{\mathbb R^d\setminus\widehat F}(q)\,dq.
\end{aligned}
\]
The last equality uses
$|B(\theta_0,2/A)|=|B(0,1)|(2/A)^d$.

Next fix $x\in\widetilde B$ and $z\in B$.
After $y=y_0+rv$, use the change of variables
\[
 q=\theta_0+\frac{u}{A}-\frac{v+w}{2A},
 \qquad dv=(2A)^d\,dq.
\]
Its image is
$B(\theta_0+u/A-w/(2A),1/(2A))$, again contained in
$B(\theta_0,2/A)$. Consequently,
\[
\begin{aligned}
 &\fint_B
 \mathbf1_{\mathbb R^d\setminus\widehat F}
 \left(x-\frac{y+z}{2}\right)\,dy\\
 &\quad=
 \frac1{|B(0,1)|}
 \int_{B(0,1)}
 \mathbf1_{\mathbb R^d\setminus\widehat F}
 \left(\theta_0+\frac{u}{A}
 -\frac{v+w}{2A}\right)\,dv\\
 &\quad=
 \frac{(2A)^d}{|B(0,1)|}
 \int_{B(\theta_0+u/A-w/(2A),\,1/(2A))}
 \mathbf1_{\mathbb R^d\setminus\widehat F}(q)\,dq\\
 &\quad\le
 \frac{(2A)^d}{|B(0,1)|}
 \int_{B(\theta_0,2/A)}
 \mathbf1_{\mathbb R^d\setminus\widehat F}(q)\,dq\\
 &\quad=
 4^d\fint_{B(\theta_0,2/A)}
 \mathbf1_{\mathbb R^d\setminus\widehat F}(q)\,dq.
\end{aligned}
\]
Since the indicator is symmetric in $y,z$, interchanging
these variables gives, for fixed $x\in\widetilde B$ and $y\in B$,
\[
 \fint_B
 \mathbf1_{\mathbb R^d\setminus\widehat F}
 \left(x-\frac{y+z}{2}\right)\,dz
 \le
 4^d\fint_{B(\theta_0,2/A)}
 \mathbf1_{\mathbb R^d\setminus\widehat F}(q)\,dq.
\]

All three bounds are independent of the two fixed variables.
Let $C_d$ be the dimensional constant in
\eqref{eq:rough-defect-support}, and set
\[
 \varepsilon_A
 =
 4^d C_d\left(1+\frac{\|\Omega\|_\infty}{\kappa}\right)
 \fint_{B(\theta_0,2/A)}
 \mathbf1_{\mathbb R^d\setminus\widehat F}(q)\,dq.
\]
Integrating \eqref{eq:rough-defect-support} in each
variable and applying the preceding bounds proves
\eqref{eq:rough-integral-smallness}.
The assertions hold for almost every choice of the
fixed variables by Fubini's theorem, which justifies
the essential suprema.

Finally, \eqref{eq:rough-density-point}, with $t=2/A$,
gives $\varepsilon_A\to0$ as $A\to\infty$.
Its definition involves neither the center nor the radius
of $B$, so the estimates are uniform over all balls.
\end{proof}

\begin{proof}[Proof of
Theorem~\ref{thm:joint-necessary-main}\textup{(i)}]
Write
\[
 N_\Omega
 =
 \|C_{\vec b}^{1,2}(T_\Omega)\|_{L^2\times L^2\to L^1},
\]
and assume $N_\Omega<\infty$.
Choose $A\geq A_0$ in Lemma~\ref{lem:rough-integral-smallness} so large that
$C_d\varepsilon_A\leq 1/2$, where $C_d$ is the dimensional
constant arising below, and fix this choice throughout the proof.

Let $Q$ be a cube with center $c_Q$, and set
\[
    B=B(c_Q,r),\qquad r=\frac{\sqrt d}{2}\ell(Q),
\]
so that $Q\subset B$ and $|B|\lesssim_d |Q|$. Let $\widetilde B$
and $k_A$ be supplied by Lemma~\ref{lem:rough-integral-smallness}. For $f_1,f_2\in L^2$
supported in $Q$ and bounded $h$ supported in $\widetilde B$, the separated-support
kernel identity is
\begin{equation}\label{eq:rough-commutator-form}
\begin{aligned}
 &\langle C_{\vec b}^{1,2}(T_\Omega)(f_1,f_2),h\rangle\\
 &\quad=
 \iiint_{\widetilde B\times Q\times Q}
 (b_1(x)-b_1(y))(b_2(x)-b_2(z))
 K_\Omega(x,y,z)f_1(y)f_2(z)h(x)
 \,dy\,dz\,dx.
\end{aligned}
\end{equation}
Moreover,
\begin{equation}\label{eq:rough-commutator-form-bound}
 |\langle C_{\vec b}^{1,2}(T_\Omega)(f_1,f_2),h\rangle|
 \le
 N_\Omega\|f_1\|_2\|f_2\|_2\|h\|_\infty.
\end{equation}

To justify \eqref{eq:rough-commutator-form} under
the form-norm convention of Section~\ref{sec2}, set
\[
E_n=\{w\in Q:|b_1(w)|+|b_2(w)|\le n\},
\qquad
\widetilde E_n
=\{w\in\widetilde B:|b_1(w)|+|b_2(w)|\le n\},
\]
and define
\[
f_{i,n}
=f_i\mathbf1_{E_n}\mathbf1_{\{|f_i|\le n\}},
\qquad
h_n=h\mathbf1_{\widetilde E_n}.
\]
These are admissible tests, so the separated-support
kernel identity holds for them and
\[
\bigl|
\langle C_{\vec b}^{1,2}(T_\Omega)
(f_{1,n},f_{2,n}),h_n\rangle
\bigr|
\le N_\Omega\|f_1\|_2\|f_2\|_2\|h\|_\infty.
\]

Separation makes $K_\Omega$ bounded on
$\widetilde B\times Q\times Q$.
Moreover, the expanded integrand is absolutely integrable:
$b_1b_2\in L^1(\widetilde B)$,
$b_i\in L^1(\widetilde B)$, and
$b_if_i,f_i\in L^1(Q)$.
Dominated convergence therefore gives the limiting
kernel pairing and proves
\eqref{eq:rough-commutator-form-bound}.
For any bounded realization agreeing with the
admissible forms, this pairing coincides with its
operator pairing.
For $i=1,2$, put
\[
 a_i=b_i-\langle b_i\rangle_Q,\qquad
 \sigma_i=\left(\fint_Q|a_i|^2\right)^{1/2}.
\]
We will estimate $\sigma_1\sigma_2$ uniformly over $Q$.
If this product is zero, there is nothing to prove.
Otherwise, define
\[
 D_i(x)=\bigl(|a_i(x)|^2+\sigma_i^2\bigr)^{1/2},
 \qquad x\in\widetilde B.
\]
Since $\fint_Q a_1=\fint_Q a_2=0$, expanding the squares
gives, for almost every $x\in\widetilde B$,
\begin{equation}\label{eq:rough-variance-identity}
\begin{aligned}
 \int_Q |b_1(x)-b_1(y)|^2\,dy
 &= |Q||a_1(x)|^2+\int_Q |a_1(y)|^2\,dy
  = |Q|D_1(x)^2,\\
 \int_Q |b_2(x)-b_2(z)|^2\,dz
 &= |Q||a_2(x)|^2+\int_Q |a_2(z)|^2\,dz
  = |Q|D_2(x)^2.
\end{aligned}
\end{equation}

Take the input tests
\[
 f_{i,0}=\mathbf1_Q,\qquad
 f_{i,1}=\frac{\overline{a_i}}{\sigma_i}\mathbf1_Q,
 \qquad i=1,2,
\]
and the output tests
\[
\begin{aligned}
 h_{00}
 &=\frac{\overline{a_1a_2}}{D_1D_2}
   \mathbf1_{\widetilde B},
 &h_{10}
 &=-\frac{\sigma_1\overline{a_2}}{D_1D_2}
   \mathbf1_{\widetilde B},\\
 h_{01}
 &=-\frac{\sigma_2\overline{a_1}}{D_1D_2}
   \mathbf1_{\widetilde B},
 &h_{11}
 &=\frac{\sigma_1\sigma_2}{D_1D_2}
   \mathbf1_{\widetilde B}.
\end{aligned}
\]
Each input has $L^2$ norm $|Q|^{1/2}$, and each output
test has $L^\infty$ norm at most one.
On $\widetilde B\times Q\times Q$,
\[
 \sum_{\mu,\nu=0}^1
 f_{1,\mu}(y)f_{2,\nu}(z)h_{\mu\nu}(x)
 =
 \frac{
 \overline{(a_1(x)-a_1(y))(a_2(x)-a_2(z))}
 }{D_1(x)D_2(x)}.
\]
Hence \eqref{eq:rough-commutator-form} gives
\begin{equation}\label{eq:rough-four-tests}
 \sum_{\mu,\nu=0}^1
 \langle
 C_{\vec b}^{1,2}(T_\Omega)(f_{1,\mu},f_{2,\nu}),
 h_{\mu\nu}
 \rangle
 =
 \iiint K_\Omega(x,y,z)W(x,y,z)\,dy\,dz\,dx,
\end{equation}
where
\[
 W(x,y,z)
 =
 \frac{
 |b_1(x)-b_1(y)|^2|b_2(x)-b_2(z)|^2
 }{D_1(x)D_2(x)}
 \ge0.
\]
Here and below, triple integrals are over
$\widetilde B\times Q\times Q$.
By \eqref{eq:rough-variance-identity},
\begin{equation}\label{eq:rough-weight-integral}
 \iiint W
 =
 |Q|^2\int_{\widetilde B}D_1(x)D_2(x)\,dx
 <\infty.
\end{equation}
The integral is finite since $D_i\in L^2(\widetilde B)$.
By \eqref{eq:rough-commutator-form-bound}, the absolute
value of the left-hand side of
\eqref{eq:rough-four-tests} is at most $4N_\Omega|Q|$.
Multiplying by $\lambda$ and taking real parts, we obtain
\begin{equation}\label{eq:rough-main-error-estimate}
\begin{aligned}
 4N_\Omega|Q|
 \ge
 k_A|Q|^2\int_{\widetilde B}D_1D_2-
 \iiint
 \bigl(k_A-\operatorname{Re}(\lambda K_\Omega)\bigr)_+W.
\end{aligned}
\end{equation}

We now apply Lemma~\ref{lem:rough-integral-smallness}
to the error term. Since $|a_i(x)|\le D_i(x)$ and
$D_i(x)\ge\sigma_i$, $i=1,2$, we have
\[
\begin{aligned}
 \frac{|a_1(x)-a_1(y)|^2}{D_1(x)}
 &\le
 2D_1(x)+\frac{2|a_1(y)|^2}{\sigma_1},\\
 \frac{|a_2(x)-a_2(z)|^2}{D_2(x)}
 &\le
 2D_2(x)+\frac{2|a_2(z)|^2}{\sigma_2}.
\end{aligned}
\]
Consequently,
\begin{equation}\label{eq:rough-weight-majorant}
\begin{aligned}
 W(x,y,z)\le4\Bigg(
 D_1(x)D_2(x)
 +D_1(x)\frac{|a_2(z)|^2}{\sigma_2}
 +D_2(x)\frac{|a_1(y)|^2}{\sigma_1}
 +\frac{|a_1(y)|^2|a_2(z)|^2}{\sigma_1\sigma_2}
 \Bigg).
\end{aligned}
\end{equation}

For the first term, integrate first in $y$ and use
the second estimate in
\eqref{eq:rough-integral-smallness}:
\[
\begin{aligned}
&\iiint_{\widetilde B\times Q\times Q}
 \bigl(k_A-\Re(\lambda K_\Omega)\bigr)_+
 D_1(x)D_2(x)\,dy\,dz\,dx  \\
&\qquad\leq
 \varepsilon_A k_A |B||Q|
 \int_{\widetilde B}D_1D_2
 \lesssim_d
 \varepsilon_A k_A |Q|^2
 \int_{\widetilde B}D_1D_2.
\end{aligned}
\]
For the second term, the same estimate gives
\[
\begin{aligned}
&\iiint_{\widetilde B\times Q\times Q}
 \bigl(k_A-\Re(\lambda K_\Omega)\bigr)_+
 D_1(x)\frac{|a_2(z)|^2}{\sigma_2}\,dy\,dz\,dx \\
&\qquad\leq
 \varepsilon_A k_A |B||Q|\sigma_2
 \int_{\widetilde B}D_1
 \lesssim_d
 \varepsilon_A k_A |Q|^2
 \int_{\widetilde B}D_1D_2.
\end{aligned}
\]
Integrating first in $z$ and using the third estimate
in \eqref{eq:rough-integral-smallness} gives the same
bound for the third term.

For the last term, integrate first in $x$ and use
the first estimate in
\eqref{eq:rough-integral-smallness}:
\[
\begin{aligned}
&\iiint_{\widetilde B\times Q\times Q}
 \bigl(k_A-\Re(\lambda K_\Omega)\bigr)_+
 \frac{|a_1(y)|^2|a_2(z)|^2}{\sigma_1\sigma_2}\,dy\,dz\,dx \\
&\qquad\leq
 \varepsilon_A k_A |B||Q|^2\sigma_1\sigma_2
 \leq
 \varepsilon_A k_A |Q|^2
 \int_{\widetilde B}D_1D_2.
\end{aligned}
\]
Combining these estimates with
\eqref{eq:rough-weight-majorant}, we obtain
\[
\iiint_{\widetilde B\times Q\times Q}
 \bigl(k_A-\Re(\lambda K_\Omega)\bigr)_+W
 \leq
 C_d\varepsilon_A k_A|Q|^2
 \int_{\widetilde B}D_1D_2.
\]
Substituting into
\eqref{eq:rough-main-error-estimate} and using
$C_d\varepsilon_A\leq1/2$ yields
\[
4N_\Omega|Q|
 \geq \frac12 k_A|Q|^2\int_{\widetilde B}D_1D_2
 \geq \frac12 k_A|Q|^2|B|\sigma_1\sigma_2.
\]
Since $k_A|B||Q|\simeq_d\kappa A^{-2d}$, we have proved
\begin{equation}\label{eq:rough-local-ball-lower}
\left(\fint_Q|b_1-\langle b_1\rangle_Q|^2\right)^{1/2}
\left(\fint_Q|b_2-\langle b_2\rangle_Q|^2\right)^{1/2}
\lesssim_d \kappa^{-1}A^{2d}N_\Omega.
\end{equation}
Applying \eqref{eq:rough-local-ball-lower} and taking
the supremum over $Q$, we obtain
$
 S_{2,2}(b_1,b_2)\lesssim_\Omega N_\Omega.
$
Finally, H\"older's inequality gives
$T_1(b_1,b_2)\le S_{2,2}(b_1,b_2)$, which proves
\eqref{eq:S22-rough-lower}.
\end{proof}

We now prove
Theorem~\ref{thm:joint-necessary-main}\textup{(ii)}.
We first establish the equality of the two
commutator norms, then prove the lower bound
for $C_{\vec b}^{1,1}(T_\Omega^{*2})$.

\begin{proof}[Proof of
Theorem~\ref{thm:joint-necessary-main}\textup{(ii)}]
By the definitions of the partial transposes
and the bilinear pairing, expansion on admissible
tests gives
\begin{equation}\label{eq:rough-repeated-duality}
\begin{aligned}
&\bigl\langle
C_{\vec b}^{1,1}(T_\Omega^{*2})(f,g),h
\bigr\rangle\\
&\quad=
\bigl\langle
T_\Omega(f,b_1b_2h)
-T_\Omega(b_1f,b_2h)
-T_\Omega(b_2f,b_1h)
+T_\Omega(b_1b_2f,h),g
\bigr\rangle\\
&\quad=
\bigl\langle
C_{\vec b}^{2,2}(T_\Omega^{*1})(g,h),f
\bigr\rangle.
\end{aligned}
\end{equation}
Taking absolute values and suprema over admissible
triples satisfying
\[
\|f\|_2\le1,\qquad
\|h\|_2\le1,\qquad
\|g\|_\infty\le1,
\]
we obtain
\[
\|C_{\vec b}^{1,1}(T_\Omega^{*2})\|_
 {L^2\times L^\infty\to L^2}
=
\|C_{\vec b}^{2,2}(T_\Omega^{*1})\|_
 {L^\infty\times L^2\to L^2}
=:N.
\]
Assume $N<\infty$. It remains to prove the lower bound
for $C_{\vec b}^{1,1}(T_\Omega^{*2})$.

Fix $A\ge A_0$ so that $\varepsilon_A\le1/2$.
For each cube $Q$, use the auxiliary balls
$B\supset Q$ and $\widetilde B$ from part~(i), and put
\[
g=\mathbf1_{\widetilde B},
\qquad
H_Q(y,z)
=\int_{\widetilde B}K_\Omega^{*2}(z,y,x)\,dx
=\int_{\widetilde B}K_\Omega(x,y,z)\,dx.
\]
The first estimate in \eqref{eq:rough-integral-smallness}
gives, for almost every $(y,z)\in Q\times Q$,
\begin{equation}\label{eq:rough-averaged-positive}
\begin{aligned}
\operatorname{Re}(\lambda H_Q(y,z))
&\ge k_A|B|
-\int_{\widetilde B}
\bigl(k_A-\operatorname{Re}[\lambda K_\Omega(x,y,z)]\bigr)_+
\,dx\\
&\ge(1-\varepsilon_A)k_A|B|\\
&\ge\frac12k_A|B|
=\frac{c_\Omega}{|Q|},
\end{aligned}
\end{equation}
where $c_\Omega:=\frac12k_A|B||Q|>0$
is independent of $Q$.

Put
\[
a_i=b_i-\langle b_i\rangle_Q,
\qquad i=1,2,
\]
and
\[
P(y,z)
=
(a_1(y)-a_1(z))(a_2(y)-a_2(z)).
\]
The assumptions imply $a_1,a_2,a_1a_2\in L^2(Q)$,
and hence $P\in L^2(Q\times Q)$.

For $f,h\in L^2(Q)$, extended by zero outside $Q$,
and $g=\mathbf1_{\widetilde B}$, the
separated-support identity is
\begin{equation}\label{eq:rough-repeated-form}
\langle
C_{\vec b}^{1,1}(T_\Omega^{*2})(f,g),h
\rangle
=
\int_Q\int_Q
H_Q(y,z)P(y,z)f(y)h(z)\,dy\,dz.
\end{equation}
For admissible tests, this follows by expanding
\eqref{eq:rough-repeated-duality}.
For general $f,h\in L^2(Q)$, use the truncations from
part~(i). The form converges by the norm bound,
and the kernel integral converges by dominated
convergence, since $H_Q\in L^\infty(Q\times Q)$ and
\[
\|P(y,z)f(y)h(z)\|_{L^1(Q\times Q)}
\le
\|P\|_{L^2(Q\times Q)}\|f\|_2\|h\|_2.
\]

Expanding $|P|^2$ and using
$\fint_Qa_1=\fint_Qa_2=0$, we obtain
\begin{equation}\label{eq:rough-repeated-variance}
\begin{aligned}
\fint_Q\fint_Q|P(y,z)|^2\,dy\,dz
&=
2\fint_Q|a_1a_2|^2
+2\left|\fint_Qa_1a_2\right|^2\\
&\quad+
2\left(\fint_Q|a_1|^2\right)
 \left(\fint_Q|a_2|^2\right)
+2\left|\fint_Qa_1\overline{a_2}\right|^2.
\end{aligned}
\end{equation}
Discarding the two nonnegative correlation terms
and using $(s+t)^2\le2s^2+2t^2$ yields
\begin{equation}\label{eq:rough-repeated-local-oscillation}
\left(\fint_Q|a_1a_2|^2\right)^{1/2}
+
\left(\fint_Q|a_1|^2\right)^{1/2}
\left(\fint_Q|a_2|^2\right)^{1/2}
\le
\frac{\|P\|_{L^2(Q\times Q)}}{|Q|}.
\end{equation}

Take
\[
\begin{aligned}
f_0&=\mathbf1_Q,
&h_0&=\overline{a_1a_2}\mathbf1_Q,\\
f_1&=\overline{a_1}\mathbf1_Q,
&h_1&=-\overline{a_2}\mathbf1_Q,\\
f_2&=\overline{a_2}\mathbf1_Q,
&h_2&=-\overline{a_1}\mathbf1_Q,\\
f_3&=\overline{a_1a_2}\mathbf1_Q,
&h_3&=\mathbf1_Q.
\end{aligned}
\]
For $(y,z)\in Q\times Q$,
\[
\sum_{j=0}^3f_j(y)h_j(z)=\overline{P(y,z)}.
\]
Thus \eqref{eq:rough-repeated-form} gives
\begin{equation}\label{eq:rough-repeated-four-tests}
\sum_{j=0}^3
\langle
C_{\vec b}^{1,1}(T_\Omega^{*2})(f_j,g),h_j
\rangle
=
\int_Q\int_Q H_Q(y,z)|P(y,z)|^2\,dy\,dz.
\end{equation}
Moreover,
\[
\begin{aligned}
\sum_{j=0}^3\|f_j\|_2\|h_j\|_2
&=
2|Q|^{1/2}\|a_1a_2\|_{L^2(Q)}
+2\|a_1\|_{L^2(Q)}\|a_2\|_{L^2(Q)}\\
&\le
2\|P\|_{L^2(Q\times Q)},
\end{aligned}
\]
by \eqref{eq:rough-repeated-local-oscillation}.

Multiplying
\eqref{eq:rough-repeated-four-tests} by $\lambda$,
taking real parts, and using
\eqref{eq:rough-averaged-positive}, we obtain
\[
\begin{aligned}
\frac{c_\Omega}{|Q|}
\|P\|_{L^2(Q\times Q)}^2
&\le
\operatorname{Re}\left[
\lambda\sum_{j=0}^3
\langle
C_{\vec b}^{1,1}(T_\Omega^{*2})(f_j,g),h_j
\rangle
\right]\\
&\le
N\sum_{j=0}^3\|f_j\|_2\|h_j\|_2\\
&\le
2N\|P\|_{L^2(Q\times Q)}.
\end{aligned}
\]
Hence
\[
\frac{\|P\|_{L^2(Q\times Q)}}{|Q|}
\le 2c_\Omega^{-1}N.
\]
Combining it with
\eqref{eq:rough-repeated-local-oscillation}
and taking the two suprema separately yields
\[
S_{2,2}(b_1,b_2)+T_2(b_1,b_2)
\le4c_\Omega^{-1}N.
\]
This proves the lower bound for
$C_{\vec b}^{1,1}(T_\Omega^{*2})$.
The norm identity established at the beginning
of the proof gives the same bound for
$C_{\vec b}^{2,2}(T_\Omega^{*1})$,
and hence proves \eqref{eq:rough-lower-11}.
\end{proof}

\subsection{Sufficient conditions}

We prove the sufficient estimates directly from sparse
forms with explicit symbol oscillations.
The commutator bounds in
\cite{GrafakosWangXue2022,DanXue2025} are formulated under
separate $\BMO$ assumptions on the symbols.
Here the local oscillation factors are retained until
H\"older's inequality is applied, so the bounds depend
on $S_{r,t}$ and $T_u$ instead of a product of separate
$\BMO$ norms.

The two-symbol sparse estimate below follows from the
localized recursion of Section~\ref{sec3.2}, applied to
the four terms in the commutator expansion.

For a cube \(Q\), write
\[
b_i^Q:=b_i-\langle b_i\rangle_Q,
\qquad i=1,2.
\]

\begin{prop}\label{prop:two-symbol-sparse}
Assume that
$
b_1,b_2\in L^3_{\loc}(\mathbb R^d).
$
Let \(s>1\) and \(f_1,f_2,f_3\in L_c^\infty(\mathbb R^d)\).  
Then
\begin{equation}\label{eq:two-symbol-sparse}
\begin{aligned}
&\bigl|
\langle
C_{\vec b}^{1,2}(T_{\Omega})(f_1,f_2),f_3
\rangle
\bigr|\\
&\quad\lesssim A_{d,s}\|\Omega\|_\infty
   \sup_{\mathcal S}\sum_{Q\in\mathcal S}|Q|\Bigl(
   \langle |f_1|\rangle_{s,Q}
   \langle |f_2|\rangle_{s,Q}
   \langle |b_1^Q b_2^Q f_3|\rangle_{s,Q}\\
&\qquad\qquad\qquad\qquad\qquad\qquad
  +\langle |f_1|\rangle_{s,Q}
   \langle |b_2^Q f_2|\rangle_{s,Q}
   \langle |b_1^Q f_3|\rangle_{s,Q}\\
&\qquad\qquad\qquad\qquad\qquad\qquad
  +\langle |b_1^Q f_1|\rangle_{s,Q}
   \langle |f_2|\rangle_{s,Q}
   \langle |b_2^Q f_3|\rangle_{s,Q}\\
&\qquad\qquad\qquad\qquad\qquad\qquad
  +\langle |b_1^Q f_1|\rangle_{s,Q}
   \langle |b_2^Q f_2|\rangle_{s,Q}
   \langle |f_3|\rangle_{s,Q}
   \Bigr).
\end{aligned}
\end{equation}
Here \(A_{d,s}\) is the constant in
Proposition~\ref{prop:local-rough-estimate}.
\end{prop}

\begin{proof}
     We first prove the estimate for $1<s\le4/3$.
Since $b_1,b_2\in L^3_{\loc}$ and
$b_1b_2\in L^{3/2}_{\loc}$, every stopping function below
belongs to $L^s(3Q)$. For $s>4/3$, the estimate at $s=4/3$
and monotonicity of normalized averages give the result,
after enlarging the dimensional implicit constant.

Use the truncation convention.
For arbitrary constants
\(c_1,c_2\in\mathbb C\), one has
\begin{equation}\label{eq:two-symbol-identity}
\begin{aligned}
&\langle
C_{\vec b}^{1,2}(T)(f_1,f_2),f_3
\rangle\\
&=\Lambda(f_1,f_2,(b_1-c_1)(b_2-c_2)f_3)
 -\Lambda((b_1-c_1)f_1,f_2,(b_2-c_2)f_3)\\
&\quad
 -\Lambda(f_1,(b_2-c_2)f_2,(b_1-c_1)f_3)
 +\Lambda((b_1-c_1)f_1,(b_2-c_2)f_2,f_3).
\end{aligned}
\end{equation}

Fix a top cube \(Q\) as in the proof of
Lemma~\ref{lem:first-order-commutator-sparse}, and set
\[
c_{i,Q}:=\langle b_i\rangle_{3Q},\qquad i=1,2.
\]
Let \(\mathcal C_Q\) denote the localized version of the left-hand side
of \eqref{eq:two-symbol-identity}, with each occurrence of
\(\Lambda\) replaced by \(\Lambda_Q\).

Apply the stopping-time construction simultaneously to
\[
f_1,\quad (b_1-c_{1,Q})f_1,
\]
\[
f_2,\quad (b_2-c_{2,Q})f_2,
\]
and
\[
f_3,\quad (b_1-c_{1,Q})f_3,\quad
(b_2-c_{2,Q})f_3,\quad
(b_1-c_{1,Q})(b_2-c_{2,Q})f_3,
\]
all restricted to \(3Q\).

For each stopping child $L\subset Q$, expand
$\mathcal C_L$ using the same constants $c_{1,Q},c_{2,Q}$.
Subtracting the child identities from the identity on $Q$
replaces each occurrence of $\Lambda_Q$ by
$\Lambda_{\mathcal P(Q)}$.

Proposition~\ref{prop:local-rough-estimate} and
\eqref{eq:stopping-control} then bound this difference
by the sum of the four products in
\eqref{eq:two-symbol-sparse}, with every averaging cube
replaced by $3Q$ and with prefactor
$A_{d,s}\|\Omega\|_\infty|Q|$.

Iterate the estimate over the stopping children, resetting
$c_{i,L}=\langle b_i\rangle_{3L}$ at each new top cube.
The recursion terminates once the local upper scale is at
most $\mu$. The resulting tree is $1/2$-sparse.
Its threefold dilation is $(2\cdot3^d)^{-1}$-sparse by the
argument following \eqref{eq:first-local-bound}.
Relabeling the dilated cubes proves
\eqref{eq:two-symbol-sparse}.
\end{proof}

We recall the following equivalent formulation of the standard
major-subset characterization of weak Lorentz spaces;
see \cite[Lemma~3.31]{MR4647944}.

\begin{lem}\label{lemm4.4}
Let $0<r<1$ and $F\in L^1_{\loc}(\mathbb R^d)$.
Then
\[
\|F\|_{L^{r,\infty}}
\simeq_r
\sup_{\substack{G\subset\mathbb R^d\ \mathrm{bounded}\\
                 0<|G|<\infty}}
\inf_{\substack{G'\subset G\\|G'|\ge|G|/2}}
|G|^{1/r-1}
\sup_{|h|\le\mathbf1_{G'}}
|\langle F,h\rangle|,
\]
where all sets are measurable.
\end{lem}

For \(1\le r_1,r_2<\infty\), set
\[
M_{r_1,r_2}(f_1,f_2)(x)
:=
\sup_{Q\ni x}
\langle |f_1|\rangle_{r_1,Q}
\langle |f_2|\rangle_{r_2,Q}.
\]

The following standard estimate follows from the pointwise bound
$M_{r_1,r_2}(f_1,f_2)\leq M_{r_1}f_1\,M_{r_2}f_2$,
the Hardy--Littlewood maximal theorem, and H\"older's inequality
(applied to the $p$-th power when $p<1$).
See also \cite{MR2483720} for the multilinear maximal framework.
\begin{lem}\label{lemm4.7}
Suppose that \(1\le r_i<p_i\le\infty\), \(i=1,2\), and
\(1/p=1/p_1+1/p_2 >0\).  Then
\[
 \|M_{r_1,r_2}(f_1,f_2)\|_{L^{p,\infty}}
 \lesssim
 \|f_1\|_{L^{p_1}}\|f_2\|_{L^{p_2}}.
\]
\end{lem}

\begin{proof}[Proof of Theorem~\ref{thmm4.5}]
Fix $\varepsilon>0$ and initially take
$f_1,f_2\in L^\infty_c(\mathbb R^d)$.
By Lemma~\ref{lem:rough-form-limit}, the resulting
commutator belongs to $L^1_{\loc}$.

\smallskip
\noindent\emph{The Banach range.}
Assume first that \(1\le p<\infty\), and put
\[
 A=p+\varepsilon,\qquad
 B=\max\{p_1'+\varepsilon,p+\varepsilon\},\qquad
 C=\max\{p_2'+\varepsilon,p+\varepsilon\}.
\]
Since \(A>p\), \(B>\max\{p_1',p\}\), and
\(C>\max\{p_2',p\}\), we may choose \(s>1\) sufficiently close to \(1\)
so that, with
\[
 \frac1{r_0}=\frac1s-\frac1A,\qquad
 \frac1{r_1}=\frac1s-\frac1B,\qquad
 \frac1{r_2}=\frac1s-\frac1C,
\]
one has
\[
 r_0<p',\qquad
 r_1<\min\{p_1,p'\},\qquad
 r_2<\min\{p_2,p'\},\qquad
 s<\min\{p_1,p_2,p'\},
\]
with the usual interpretation when \(p'=\infty\).

Let \(\mathcal R_1,\ldots,\mathcal R_4\) denote the four sparse sums in
\eqref{eq:two-symbol-sparse}.  H\"older's inequality inside each cube,
the disjoint major subsets of the sparse family, and the boundedness of
the maximal operators give
\begin{align*}
\mathcal R_1
&\lesssim T_A(b_1,b_2)
 \|M_s f_1\|_{L^{p_1}}
 \|M_s f_2\|_{L^{p_2}}
 \|M_{r_0}f_3\|_{L^{p'}},\\
\mathcal R_2
&\lesssim S_{B,C}(b_1,b_2)
 \|M_s f_1\|_{L^{p_1}}
 \|M_{r_2}f_2\|_{L^{p_2}}
 \|M_{r_1}f_3\|_{L^{p'}},\\
\mathcal R_3
&\lesssim S_{B,C}(b_1,b_2)
 \|M_{r_1}f_1\|_{L^{p_1}}
 \|M_s f_2\|_{L^{p_2}}
 \|M_{r_2}f_3\|_{L^{p'}},\\
\mathcal R_4
&\lesssim S_{B,C}(b_1,b_2)
 \|M_{r_1}f_1\|_{L^{p_1}}
 \|M_{r_2}f_2\|_{L^{p_2}}
 \|M_s f_3\|_{L^{p'}}.
\end{align*}
Consequently, by \eqref{eq:two-symbol-sparse} and duality,
\begin{equation}\label{eq:banach-sufficient}
\|C_{\vec b}^{1,2}(T_\Omega)(f_1,f_2)\|_{L^p}
\lesssim A_{d,s}\|\Omega\|_\infty
 \bigl(T_{p+\varepsilon}(b_1,b_2)+S_{B,C}(b_1,b_2)\bigr)
 \prod_{i=1}^2\|f_i\|_{L^{p_i}}.
\end{equation}

\smallskip
\noindent\emph{The quasi-Banach range.}
Assume now that $1/2<p<1$.
Then both input exponents are finite. Set
\[
 B=p_1'+\varepsilon,\qquad C=p_2'+\varepsilon.
\]
Choose \(s>1\) sufficiently close to \(1\) so that
\begin{equation}\label{eq:quasi-parameter-choice}
\begin{aligned}
s<1+\varepsilon,\qquad s<p_i,
\qquad \frac{s p_i}{p_i-s}<p_i'+\varepsilon
,\qquad i=1,2.
\end{aligned}
\end{equation}
Define \(r_1,r_2\) by
\[
 \frac1{r_1}=\frac1s-\frac1B,\qquad
 \frac1{r_2}=\frac1s-\frac1C.
\]
Then \(s<r_i<p_i\).  The choice in
\eqref{eq:quasi-parameter-choice} is possible because
\[
 \lim_{s\rightarrow1}\frac{s p_i}{p_i-s}=p_i'.
\]

Fix a bounded measurable set $G$ with $0<|G|<\infty$.
By Lemma~\ref{lemm4.4}, it suffices to find
$G'\subset G$, $|G'|\ge|G|/2$, for which the required
pairing estimate holds uniformly over
$|h|\le\mathbf1_{G'}$.
Put
\[
 \mathcal M=M_{r_1,r_2}(f_1,f_2),\qquad
 N=\|\mathcal M\|_{L^{p,\infty}}.
\]
If \(N=0\), then \(\mathcal M=0\) almost everywhere and the desired
estimate is immediate. Hence we may assume that \(N>0\).
For a sufficiently large constant $c_{d,p}$, let
\[
E=\{x:\mathcal M(x)>c_{d,p}|G|^{-1/p}N\},
\qquad
F=\{x:M\mathbf1_E(x)\ge\tfrac12\}.
\]
By Lemma~\ref{lemm4.7} and the weak type \((1,1)\) estimate for \(M\),
choosing \(c_{d,p}\) sufficiently large gives
$
|F|\leq \frac12|G|.
$
Set
$
G'=G\setminus F.
$
This choice of \(G'\) is independent of the sparse family.

Let \(h\) satisfy \(|h|\le\mathbf1_{G'}\).  
Since $s<1+\varepsilon$, H\"older's inequality with
conjugate exponents
\[
\frac{1+\varepsilon}{s}
\quad\text{and}\quad
\frac{1+\varepsilon}{1+\varepsilon-s}
\]
gives
\[
 \langle |b_1^Q b_2^Q h|\rangle_{s,Q}
 \le T_{1+\varepsilon}(b_1,b_2)
 \langle\mathbf1_{G'}\rangle_{
 \frac{s(1+\varepsilon)}{1+\varepsilon-s},Q}.
\]
For the other three forms, H\"older's inequality with the pairs
\((B,r_1)\) and \((C,r_2)\) gives the factor
\(S_{B,C}(b_1,b_2)\).  Since every average of \(\mathbf1_{G'}\) is at
most \(1\), all four forms are bounded by
\begin{equation}\label{eq:four-forms-common-bound}
 \bigl(T_{1+\varepsilon}(b_1,b_2)+S_{B,C}(b_1,b_2)\bigr)
 \sum_{Q\in\mathcal S_*}
 \langle |f_1|\rangle_{r_1,Q}
 \langle |f_2|\rangle_{r_2,Q}|Q|,
\end{equation}
where \(\mathcal S_*\) consists of the cubes whose original summand is
nonzero.  In particular, every \(Q\in\mathcal S_*\) meets \(G'\), and
hence
\begin{equation}\label{eq:active-cubes-good}
 |Q\cap E|<\frac12|Q|,\qquad |Q|\le2|Q\cap E^c|.
\end{equation}
First suppose that $\mathcal S_*$ is finite.
Write
\[
a_Q=\langle|f_1|\rangle_{r_1,Q}
    \langle|f_2|\rangle_{r_2,Q},
\]
and let $\kappa>0$ be the fixed sparsity parameter.
For $t>0$, set
\[
\mathcal A_t=\{Q\in\mathcal S_*:a_Q>t\},
\qquad U_t=\bigcup_{Q\in\mathcal A_t}Q.
\]
Disjoint sparse witnesses give
\[
\sum_{Q\in\mathcal A_t}|Q|\le\kappa^{-1}|U_t|.
\]
By a greedy selection in decreasing order of side length,
choose a finite pairwise disjoint family
$\{P_\ell\}_\ell\subset\mathcal A_t$
such that
$U_t\subset\bigcup_\ell3P_\ell$.
Using \eqref{eq:active-cubes-good}, we obtain
\[
\begin{aligned}
|U_t|
\le3^d\sum_\ell|P_\ell|
\le2\cdot3^d\sum_\ell|P_\ell\cap E^c|
\le2\cdot3^d|E^c\cap\{\mathcal M>t\}|.
\end{aligned}
\]
The last inequality holds because $a_{P_\ell}>t$ implies
$P_\ell\subset\{\mathcal M>t\}$.
Integrating in $t$ yields
\begin{equation}\label{eq:principal-cube-estimate}
\begin{aligned}
\sum_{Q\in\mathcal S_*}a_Q|Q|
&=\int_0^\infty\sum_{Q\in\mathcal A_t}|Q|\,dt\\
&\le\frac{2\cdot3^d}{\kappa}
\int_0^\infty|E^c\cap\{\mathcal M>t\}|\,dt\\
&=\frac{2\cdot3^d}{\kappa}\int_{E^c}\mathcal M.
\end{aligned}
\end{equation}
The general case follows by monotone convergence.

By the definition of \(E\) and the weak-\(L^p\) distribution estimate,
\begin{align*}
\int_{E^c}\mathcal M
&\le N^p\int_0^{c_{d,p}|G|^{-1/p}N}t^{-p}\,dt\\
&\lesssim_p |G|^{1-1/p}N.
\end{align*}
Combining this with Lemma~\ref{lemm4.7},
\eqref{eq:two-symbol-sparse}, \eqref{eq:four-forms-common-bound}, and
\eqref{eq:principal-cube-estimate}, we obtain, uniformly for
\(|h|\le\mathbf1_{G'}\),
\begin{align*}
|\langle C_{\vec b}^{1,2}(T_\Omega)(f_1,f_2),h\rangle|
&\lesssim A_{d,s}\|\Omega\|_\infty
 \bigl(T_{1+\varepsilon}(b_1,b_2)
      +S_{p_1'+\varepsilon,p_2'+\varepsilon}(b_1,b_2)\bigr)\\
&\qquad \qquad\times |G|^{1-1/p}
 \|f_1\|_{L^{p_1}}\|f_2\|_{L^{p_2}}.
\end{align*}
Lemma~\ref{lemm4.4} now yields
\begin{equation}\label{eq:quasi-sufficient}
\|C_{\vec b}^{1,2}(T_\Omega)(f_1,f_2)\|_{L^{p,\infty}}
\lesssim A_{d,s}\|\Omega\|_\infty
 \bigl(T_{1+\varepsilon}(b_1,b_2)
      +S_{p_1'+\varepsilon,p_2'+\varepsilon}(b_1,b_2)\bigr)
 \prod_{i=1}^2\|f_i\|_{L^{p_i}}.
\end{equation}

The preceding estimates prove the asserted bounds for
bounded compactly supported inputs. When both input
exponents are finite, extension follows by density.

Suppose now that $p_2=\infty$. Then $p=p_1>1$.
Fix bounded compactly supported functions $f_1$ and $h$.
On every fixed cube, the functional
$
f_2\longmapsto
\langle C_{\vec b}^{1,2}(T_\Omega)(f_1,f_2),h\rangle
$
is continuous in the $L^6$ norm of $f_2$.

Indeed, local H\"older estimates place the
four expanded terms in the exponent triples
$(6,6,3/2)$ and $(3,6,2)$, up to permutations.
Lemma~\ref{lem:rough-form-limit} therefore
gives local $L^6$ continuity in $f_2$.
It is therefore represented by a locally integrable
function $\varphi$, with the local representations agreeing
on overlaps.

The bound already proved for bounded compactly supported
$f_2$ gives
\[
\left|\int_{\mathbb R^d}\varphi f_2\right|
\lesssim
A_{d,s}\|\Omega\|_\infty
\Bigl(
T_{p+\varepsilon}(b_1,b_2)
+
S_{\max\{p'+\varepsilon,p+\varepsilon\},\,p+\varepsilon}
(b_1,b_2)
\Bigr)
\|f_1\|_{p_1}\|h\|_{p'}\|f_2\|_\infty.
\]
Testing with bounded compactly supported approximations
to the complex sign of $\varphi$ shows that
$\varphi\in L^1(\mathbb R^d)$, with the corresponding
$L^1$ bound. Thus the pairing extends to all
$f_2\in L^\infty$ by integration against $\varphi$.
Duality in $L^p$, followed by density in the first input,
gives the claimed operator on
$L^{p_1}\times L^\infty$.
This extension is weak-$\ast$ continuous in the second
input at the level of scalar pairings.
The case $p_1=\infty$ is identical.

This completes the proof.
\end{proof}

\begin{proof}[Proof of
Theorem~\ref{thm:diagonal-sufficient-main}]
Use the truncation convention.
The first-position commutator satisfies
\begin{align*}
&\langle
C_{\vec b}^{1,1}(T)(f_1,f_2),f_3
\rangle\\
&=\Lambda(f_1,f_2,(b_1-c_1)(b_2-c_2)f_3)
 -\Lambda((b_2-c_2)f_1,f_2,(b_1-c_1)f_3)\\
&\quad
 -\Lambda((b_1-c_1)f_1,f_2,(b_2-c_2)f_3)
 +\Lambda((b_1-c_1)(b_2-c_2)f_1,f_2,f_3).
\end{align*}
Repeating the proof of Proposition~\ref{prop:two-symbol-sparse} with
the identity above, and enlarging the stopping data so as to include
every function appearing in the four triples below, yields the
corresponding four sparse products.
\[
\begin{gathered}
(f_1,f_2,b_1^Q b_2^Q f_3),\qquad
(b_2^Q f_1,f_2,b_1^Q f_3),\\
(b_1^Q f_1,f_2,b_2^Q f_3),\qquad
(b_1^Q b_2^Q f_1,f_2,f_3).
\end{gathered}
\]
Here, as usual, \(b_k^Q=b_k-\langle b_k\rangle_Q\).  
For \(i=2\), the same argument starts from the corresponding
four-term identity and yields the sparse products associated with
\[
\begin{gathered}
(f_1,f_2,b_1^Q b_2^Q f_3),\qquad
(f_1,b_2^Q f_2,b_1^Q f_3),\\
(f_1,b_1^Q f_2,b_2^Q f_3),\qquad
(f_1,b_1^Q b_2^Q f_2,f_3).
\end{gathered}
\]

Suppose first that $p\ge1$ and put $m=m_1$.
Since $m>\max\{p,p_1'\}$, we may choose $s>1$
close enough to $1$ that
\[
\frac1r=\frac1s-\frac1m,
\qquad
r<\min\{p_1,p'\},
\qquad
s<\min\{p_1,p_2,p'\}.
\]
H\"older's inequality in each cube bounds the sum of the four sparse
products by
\[
\bigl(T_m(b_1,b_2)+S_{m,m}(b_1,b_2)\bigr)
\sum_{Q\in\mathcal S}
\langle|f_1|\rangle_{r,Q}
\langle|f_2|\rangle_{s,Q}
\langle|f_3|\rangle_{r,Q}|Q|.
\]
The sparse embedding and duality prove
\eqref{eq:diagonal-sufficient} in this range.

If \(1/2<p<1\), then \(m_1=p_1'+\varepsilon\).  Choose \(s>1\) sufficiently
close to \(1\) so that \(s<1+\varepsilon\), \(s<p_2\), and
\[
 \frac1r=\frac1s-\frac1{m_1}>\frac1{p_1}.
\]
The first sparse product is controlled by
\(T_{1+\varepsilon}\le T_{m_1}\), the fourth by \(T_{m_1}\), and the two
mixed products by \(S_{m_1,m_1}\).  
After dropping the averages of the indicator of the major subset, all
four sparse forms are therefore bounded by
\[
\bigl(T_{m_1}(b_1,b_2)+S_{m_1,m_1}(b_1,b_2)\bigr)
\sum_{Q\in\mathcal S_*}
\langle |f_1|\rangle_{r,Q}
\langle |f_2|\rangle_{s,Q}|Q|.
\]
Since \(r<p_1\) and \(s<p_2\), the packing argument in the
proof of Theorem~\ref{thmm4.5}, applied to
\(M_{r,s}(f_1,f_2)\), gives the required weak-\(L^p\) estimate.

For \(i=2\), the same argument is applied with \(f_1,p_1\) and
\(f_2,p_2\) interchanged.  This replaces \(m_1\) by
\(m_2=\max\{p+\varepsilon,p_2'+\varepsilon\}\) 
and proves the corresponding estimate for bounded
compactly supported inputs.

When both input exponents are finite, the extension
follows by density.
If one input exponent is infinite, use the
scalar-pairing construction at the end of the proof
of Theorem~\ref{thmm4.5}.
The only additional local integrability pattern occurs
when both symbols multiply the varying input.
On a fixed cube, H\"older's inequality gives
\[
\|b_1b_2f_i\|_{6/5}
\le
\|b_1\|_3\|b_2\|_3\|f_i\|_6,
\qquad
\frac{5}{6}+\frac{1}{12}+\frac{1}{12}=1.
\]
Thus Lemma~\ref{lem:rough-form-limit}, with exponent
triple $(6/5,12,12)$ and the appropriate permutation,
gives local $L^6$ continuity in this case as well.
The resulting scalar functional is represented by
an $L^1$ function, by the already proved uniform
$L^\infty$ bound.
This yields the required extension to a general
$L^\infty$ input.
\end{proof}

\bibliographystyle{plain}
\bibliography{ams}

\end{document}